\documentclass[11pt, a4paper]{amsart}
\usepackage{a4wide}
\usepackage{graphicx,enumitem}
\usepackage{amssymb}
\usepackage{color}
\usepackage{ulem}
\usepackage{latexsym}
\usepackage{hyperref}

\allowdisplaybreaks
\numberwithin{equation}{section}

\def\eps{\varepsilon}

\newcommand{\R}{\mathbb{R}}
\newcommand{\IR}{\R}
\newcommand{\N}{\mathbb{N}}
\newcommand{\IN}{\N}
\newcommand{\IC}{\mathbb{C}}

\newcommand{\cA}{\mathcal{A}}
\newcommand{\cB}{\mathcal{B}}

\newcommand{\cD}{\mathcal{D}} %Domains of Operators
\newcommand{\cE}{\mathcal{E}}
\newcommand{\cF}{\mathcal{F}}

\newcommand{\cL}{\mathcal{L}}

\newcommand{\cS}{\mathcal{S}}

\newcommand{\rD}{{\mathrm{D}}}
\DeclareMathOperator{\spann}{\mathrm{span}}

\newcommand{\be}{\begin{equation}}
\newcommand{\ee}{\end{equation}}

\newcommand{\dd}{{\mathrm d}}
\newcommand{\matr}[1]{{\bf #1}}

\newtheorem{lemma}{Lemma}[section]
\newtheorem{proposition}[lemma]{Proposition}

\newtheorem{theorem}[lemma]{Theorem}
\newtheorem{remark}[lemma]{Remark}
\newtheorem{example}[lemma]{Example}

\newtheorem{assumption}[lemma]{Assumption}
\newtheorem{corollary}[lemma]{Corollary}
\theoremstyle{remark}

\theoremstyle{definition}
\newtheorem{definition}[lemma]{Definition}
\definecolor{darkred}{rgb}{.6,0,0}

\definecolor{darkgreen}{rgb}{0,.6,0}

\newcommand{\HH}{\mathrm{H}}

\newcommand{\II}{{\mathrm I}}

\begin{document}
\title[Compressive Space-Time Galerkin for Parabolic PDEs]{
Compressive Space-Time Galerkin Discretizations of Parabolic Partial
Differential Equations}

\author[S.~Larsson]{Stig Larsson} \address[Stig Larsson]{\newline
  Department of Mathematical Sciences
\newline Chalmers University of Technology \newline 
and University of Gothenburg\newline
SE--412 96 Gothenburg, Sweden.} \email[]{stig@chalmers.se}

\author[Ch.~Schwab]{Christoph Schwab} \address[Christoph Schwab]{
\newline 
Seminar f\"ur Angewandte Mathematik 
\newline
ETH Z\"urich \newline R\"amistrasse 101, CH--8092 Z\"urich, Switzerland.} 
\email[]{schwab@math.ethz.ch}

\thanks{Acknowledgement. 
The work of Christoph Schwab was supported in part by ERC AdG no.~247277. 
}
\date{\today}



\begin{abstract}
We study linear parabolic initial-value problems in a space-time
variational formulation based on fractional calculus. 
This formulation %, introduced by Fontes in \cite{FDiss},
uses ``time derivatives of order one half'' on the bi-infinite 
time axis.  
We show that for linear, parabolic initial-boundary 
value problems on $(0,\infty)$,
the corresponding bilinear form admits an inf-sup condition 
with sparse tensor product trial and test function spaces.
We deduce optimality of compressive, space-time Galerkin discretizations,
%along the lines of \cite{ScSt09}, 
where stability of Galerkin approximations
is implied by the well-posedness of the parabolic operator equation.
The variational setting adopted here admits more general
Riesz bases than previous work; in particular, 
\emph{no stability in negative order Sobolev spaces 
on the spatial or temporal domains} 
is required of the Riesz bases accommodated
by the present formulation.
The trial and test spaces are based on Sobolev spaces of equal order 
$1/2$ with respect to the temporal variable.
Sparse tensor products of multi-level decompositions of
the spatial and temporal spaces in Galerkin discretizations
lead to large, non-symmetric linear systems of equations.
We prove that their condition numbers
are uniformly bounded with respect to the discretization level.
In terms of the total number of degrees of freedom, 
the convergence orders equal, up to logarithmic terms, 
those of best $N$-term approximations
of solutions of the corresponding elliptic problems.
\end{abstract}

\newpage
\maketitle
\newpage
\tableofcontents
\setcounter{page}{0}
\newpage
\section{Introduction}
\label{sec:intro}
For a bounded linear and self-adjoint operator $A\in \cL(V,V^*)$ in an
evolution triplet $V\subset H \simeq H^* \subset V^*$, and a bounded
domain $D\subset \IR^n$, we consider the initial boundary value
problem for abstract, linear parabolic evolution equations
\be\label{eq:ParProb}
Bu := \partial_t u + Au = f \quad\mbox{in}\quad \IR_> = (0,\infty)\;,
\ee
with homogeneous initial condition
\be\label{eq:ParIC} u(0) = 0 \;.  
\ee 
In \eqref{eq:ParProb}, we think of $A$ as linear,
strongly elliptic (pseudo)differential operator of order $2m > 0$, 
and of $V$ as a closed subspace of $H^m(D)$ supporting homogeneous, 
essential boundary conditions of the initial boundary value problem 
\eqref{eq:ParProb}, \eqref{eq:ParIC}.

Optimality of adaptive variational space-time Galerkin discretizations
of \eqref{eq:ParProb}, \eqref{eq:ParIC} on $(0,T)$ for $T<\infty$ were
shown for the first time in \cite{ScSt09}. There, well-posedness of
suitable space-time variational saddle-point formulations of the
parabolic initial boundary value problems \eqref{eq:ParProb},
\eqref{eq:ParIC} were established.  
By means of tensorized Riesz bases
of the Bochner spaces which underlie the space-time variational
formulations, the parabolic initial boundary value problems were
converted to equivalent bi-infinite matrix problems. 
These matrix
problems were subsequently solved numerically, in optimal complexity,
by means of adaptive wavelet discretizations from \cite{CDD2}.  
We note that {\em adaptive wavelet techniques} from \cite{CDD2} were
essential in the algorithms in \cite{ScSt09}, since it used
the paradigm ``stability by adaptivity'' from \cite{CDD2}.  In
particular in \cite{ScSt09}, no stability result for nonadaptive
discretizations could be obtained, but rather followed from the
well-posedness of the infinite-dimensional problem, the Riesz basis
property and certain optimality properties of the adaptive Galerkin
discretizations (``stability by adaptivity'').

In the present paper, building on fractional calculus techniques
pioneered in variational formulations of parabolic initial boundary
value problems by M.~Fontes \cite{FDiss,F00}, we propose a space-time
variational formulation based on bilinear forms, which are,
unlike the formulations considered in \cite{ScSt09}, {\em
  ``symmetric''} in the sense that trial and test spaces, which arise
in the variational formulation, are Sobolev spaces of equal orders
with respect to time differentiation.  Stability (in the sense that a
discrete inf-sup condition holds) of our space-time Galerkin
discretization requires that the finite-dimensional trial and test
spaces are different.  

The presently considered space-time variational formulation admits a
unique variational solution in a Bochner space $X$, which is
intermediate to the solution spaces which are obtained by the
``classical'' approach.  Moreover, as shown by M.~Fontes in
\cite{F00,Fon09}, the presently considered solutions can be obtained
by monotone operator methods and, therefore, Galerkin approximations
are well-defined and stable with any closed subspaces, including in
particular sparse tensor products of multilevel hierarchies in
space and time.
It is interesting to note that time derivatives of order
$1/2$ were used already in \cite{BaiocchiBrezzi,LangWolf13} 
in order to prove error estimates in the
$X$-norm for finite element approximations of
\eqref{eq:ParProb}--\eqref{eq:ParIC}.

As in \cite{ScSt09}, we establish in the present paper
quasi-optimality of linear and nonlinear space-time adaptive and
compressive Galerkin discretizations in the space-time cylinder.  To
this end, we show a discrete inf-sup condition in the present paper,
for a suitable sparse tensor space-time Petrov-Galerkin
discretization.  The use of wavelet-type Riesz bases in space and time
then results in uniformly bounded condition numbers of the
finite-dimensional problems; notably, this holds without the Riesz
basis property in $V^*$ of the spatial wavelet basis $\Sigma$, which
was essential in \cite{ScSt09}.  In the
presently considered variational formulation, we consider in
particular long-term evolution, i.e., the time interval $(0,T)$ with
$T = \infty$, and analyze space-time compressive and adaptive
numerical approximation of long-time integration for these problems.
Unlike \cite{Andreev13,ScSt09}, we obtain stability, multilevel
preconditioning and space-time compressibility even without
adaptivity, and with trial and test spaces of equal dimension (albeit
being possibly different so that we consider a
Petrov-Galerkin formulation as in \cite{Andreev13}).
Moreover, the optimality results in Section~\ref{sec:Adap} entail
optimal, adaptive and space-time compressive methods for long-time
integration (i.e., $T=\infty$) for parabolic evolution problems.

The outline of this paper is as follows: in Section~\ref{sec:Prels},
we present basic definitions and facts from functional analysis and
fractional calculus.  In Section~\ref{sec:PDEs}, we present the
space-time variational formulation of \eqref{eq:ParProb},
\eqref{eq:ParIC}.  Section~\ref{sec:WavGalDisc}, we consider
compressive space-time Galerkin discretization with sparse tensor
subspaces.  Section~\ref{sec:Adap} addresses the space-time {\em
  adaptive} discretization of the variational formulation in Section
\ref{sec:PDEs} and establishes optimality.  The analysis in
Sections~\ref{sec:PDEs}--\ref{sec:Adap} is developed for long-time
integration, i.e., for $T=\infty$.
\section{Preliminaries}
\label{sec:Prels}
\subsection{Functional analysis}
\label{sec:FA}
We require some tools from functional analysis.  Throughout this paper
all vector spaces are real unless explicitly stated
otherwise. Consider two Banach spaces $X$ and $Y$ and a bilinear form
$\cB\colon X\times Y\rightarrow \IR$, which is bounded, i.e., there
exists a constant $C$ such that
\be\label{eq:BilCont}
| \cB(w,v) | \leq C \| w \|_X \| v \|_Y 
\quad \forall w\in X\;,\;v\in Y \;.
\ee
We are interested in solving the linear, variational problem:
for each $F\in Y^*$, find a unique $u\in X$ such that
\be\label{eq:VarProb}
\cB(u,v) = F(v) \quad \forall v\in Y\;.
\ee
The form $\cB(\cdot,\cdot)$ induces in a one-to-one fashion a bounded,
linear operator $B\in \cL(X,Y^*)$ via
$$
_{Y^*}\langle Bw, v \rangle_Y = \cB(w,v) \quad \forall w\in X\;,\;v\in Y\;,
$$
so that the unique solvability of \eqref{eq:VarProb} is related to 
the question of bounded invertibility of the operator $B\in \cL(X,Y^*)$.
There holds:
\begin{proposition}
\label{prop:infsup}
Let $X,Y$ be Banach spaces; $Y$ reflexive. 
Let $\cB\colon
X\times Y\to\IR$ be a bounded bilinear form and consider the inf-sup
condition:
\be\label{eq:infsup}
\inf_{0\ne w\in X} \sup_{0\ne v \in Y} 
\frac{\cB(w,v)}{\| w \|_X \| v\|_Y}  \geq \gamma > 0
\;,
\ee
and the (adjoint) injectivity condition:
\be\label{eq:injec}
\sup_{w\in X} \cB(w,v) > 0 \quad \forall 0\ne v\in Y\;.
\ee
The conditions \eqref{eq:infsup}--\eqref{eq:injec} hold if and only if
for each $F\in Y^*$, the variational problem \eqref{eq:VarProb} admits
a unique solution $u\in X$ and in this case there holds the estimate
\begin{align*}
\| u \|_X \leq \frac{1}{\gamma} \| F \|_{Y^*} \;.
\end{align*}
In other words, \eqref{eq:infsup}--\eqref{eq:injec} hold if and
only if the corresponding operator $B\in \cL(X,Y^*)$ is boundedly
invertible, in which case $\| B^{-1} \|_{\cL(Y^*,X)} \leq \gamma^{-1}$.
\end{proposition}
The Proposition~\ref{prop:infsup} was used in \cite{ScSt09} in
verifying that space-time saddle point formulations of
\eqref{eq:ParProb} are well-posed.  
Below, we shall be interested in
the following special case where $X=Y$.
\begin{corollary}\label{coro:X=Y}
  Assume that $X$ is a reflexive Banach space, and that the bounded
  bilinear form $\cB\colon X\times X\to\IR$ is {\em
    coercive-equivalent}, i.e., there exists an isomorphism $S\in
  \cL(X,X)$ such that $\cB(\cdot,S\cdot)$ is coercive,
  i.e., there exists $c>0$ such that
\be\label{eq:TAScoer}
\cB(w,Sw)\ge c\|w\|_X^2
\quad \forall w\in X \;.
\ee
Then the corresponding operator $B\in \cL(X,X^*)$ is boundedly
invertible.
\end{corollary}
\begin{proof} 
We assume \eqref{eq:TAScoer} and verify conditions
\eqref{eq:infsup}--\eqref{eq:injec} in Proposition
\ref{prop:infsup}.  For $0\ne w\in X$, we have $\| Sw \|_X \leq c_S
\| w \|_X$ and $Sw\ne 0$, since $S$ is an isomorphism. 
Together with \eqref{eq:TAScoer} this leads to
\begin{align*}
\sup_{0\not=v\in X}\frac{\cB(w,v)}{\|v\|_X}
\ge \frac{\cB(w,Sw) }{\|Sw\|_X}
\geq c \frac{\| w \|_X^2}{\|Sw\|_X} \geq \frac{c}{c_S} \| w \|_X \;.
\end{align*}
This proves \eqref{eq:infsup}.  
To verify \eqref{eq:injec} we compute
\begin{align*}
  \sup_{w\in X}\cB(w,v)\ge \cB(S^{-1}v,v)=\cB(S^{-1}v,S(S^{-1}v))\ge c\|S^{-1}v\|_X^2
\geq \frac{c}{c_S^2} \|v\|_X^2>0
\end{align*}
for $0\not=v\in X$.
\end{proof}
\subsection{The elliptic operator}
\label{sec:FuncSpc}
We let $(H, \langle \cdot,\cdot \rangle_H)$ and $(V, \langle
\cdot,\cdot \rangle_V)$ denote two separable Hilbert spaces with dense
embedding $V\subset H$ and duals $H^*$ and $V^*$.  We identify
$H\simeq H^*$ according to the Riesz representation theorem and obtain
the Gel'fand triple
\begin{align*}
V\subset H \simeq H^* \subset V^*\;,
\end{align*}
again with dense injections.  Let $A\in \cL(V,V^*)$ be a bounded
self-adjoint linear operator such that the corresponding bilinear form
$a(v,w)={_{V^*}\langle Av,w\rangle_{V}}$ is coercive and bounded on
$V\times V$, i.e., for some $0<\lambda_- \leq \lambda_+ < \infty$,
\be\label{eq:AssA}
a(v,v)\geq \lambda_- \| v \|_V^2
\;,\quad 
|a(v,w)|\leq \lambda_+ \|v\|_V \| w\|_V \;.
\ee
\begin{example}\label{ex:Expl1}
In a bounded Lipschitz domain $D\subset \IR^n$ 
of dimension $n\geq 1$, 
we consider the linear, second order divergence form operator
given for $v\in C_0^\infty(D)$ by
\begin{align*}
Av = -\nabla\cdot(a(x) \nabla v) + c(x) v 
\;.  
\end{align*}
Here, $a\in (L^\infty(D))^{n\times n}_\mathrm{sym}$
and $c\in L^\infty(D)$ satisfy the 
ellipticity conditions
$$
\exists \gamma > 0 \;\; \forall \xi\in \IR^n:\quad 
\xi^\top a \,\xi \geq \gamma |\xi|^2 
\;, \quad 
{\rm ess}\inf_{x\in D} c(x) \geq 0 
\;.
$$
In this case $V=H^1_0(D)$, $H=L^2(D)$, 
$a(v,w)=(a\nabla v, \nabla w)+(cv,w)$, 
and \eqref{eq:AssA} is valid.
\end{example}
\begin{example} \label{ex:Expl2}
With $D$ as in Example~\ref{ex:Expl1},
we consider the Stokes equation. 
Then
\begin{align*}
H &= \{ v\in L^2(D)^n: \operatorname{div} v = 0 \;\;\mbox{in}\;\; L^2(D)\;,
       \gamma_0(v\cdot n) = 0 \;\;\mbox{in}\;\; H^{-\frac12}(\partial D) \}\;,
\\
V &=  \{ v\in H^1_0(D)^n: \operatorname{div} v = 0 \;\;\mbox{in}\;\; L^2(D)\}\;,
\end{align*}
where $\gamma_0$ denotes the trace operator 
and the bilinear form is given by 
$a(w,v) = \int_D \nabla w : \nabla v \,\dd x$.
\end{example}
\subsection{Bochner spaces}
\label{sec:BochSpc}
We require Bochner spaces of vector-valued functions defined on
intervals. 
For an interval $I$, a Banach space $X$ with norm 
$\| \cdot \|_X$, 
and for $1\leq p \leq \infty$, we denote by $L^p(I;X)$
the space of strongly measurable functions $u\colon I\mapsto X$ 
such that
\begin{align*}
\| u \|_{L^p(I;X)} 
=
\left( \int_I \| u(t) \|_X^p \,\dd t \right)^{1/p} <\infty
\end{align*}
for $1\leq p < \infty$ with the usual modification for $p=\infty$.
Similarly, we denote by $H^1(I;X)$ the space of functions whose
distributional time derivative belongs to $L^2(I;X)$. 
We also need spaces of continuous functions: 
for $k\in \IN_0$, we denote by $C^k(\bar{I};X)$
the Banach space of $k$-times continuously differentiable 
and bounded mappings 
$u\colon \bar{I}\mapsto X$ endowed with the
standard norm $\| \cdot \|_{C^k(\bar{I};X)}$.
\subsection{Interpolation spaces}
\label{sec:Interp}
We repeatedly use assorted facts from the theory of function space
interpolation (see, e.g., \cite{BL,LM1,Triebel}). 
In particular, we use the interpolation spaces 
$[X,Y]_s$, $0<s<1$,
between two Hilbert spaces with dense embedding $X\subset Y$, 
as defined, for example, in \cite[Chap.~1, D\'ef.~2.1]{LM1}.  
% Note: the French edition has different numbers than the English!!

For $0<T\leq\infty$ we denote by $I=(-T,T)$ the symmetric interval,
with $I=\IR$ implied if $T=\infty$, and set $I_> = I\cap \{t>0\}$.
For a separable Hilbert space $H$, we define
\begin{align*}
H^1_{0,\{0\}}(I_>;H):= \{ v\in H^1(I_>;H) : v(0) = 0 \} \;. 
\end{align*}
By the continuity of the embedding $H^1(I_>;H) \subset
C^0(\overline{I_>};H)$, the set $H^1_{0,\{0\}}(I_>;H)$ is the
null space of the trace operator at $t=0$ and, therefore, a
norm-closed, linear subspace of $H^1(I_>;H)$.
We introduce the interpolation spaces
\begin{align*}
\begin{aligned}
H^{s}(I;H) &:= [L^2(I;H),H^1(I;H)]_{{s}}\;, \quad &&s\in (0,1)\;,
\\
H^{s}(I_>;H) &:= [L^2(I_>;H),H^1(I_>;H)]_{{s}}\;, \quad &&s\in (0,1)\;,
\\
H^{s}_{0,\{0\}}(I_>;H) 
&:= [L^2(I_>;H),H^1_{0\{0\}}(I_>;H)]_{s} \;, 
  \quad &&s\in (0,1)\setminus\{\tfrac12\}\;,
\\
H^{\frac12}_{00,\{0\}}(I_>;H) 
&:= [L^2(I_>;H),H^1_{0,\{0\}}(I_>;H)]_{\frac12} \;. \quad && \ 
\end{aligned}
\end{align*}
\begin{remark} \label{rmk:Hs}
With $I_> = (0,T)$ for $0<T\leq \infty$ there holds:
 
\noindent {\upshape (1)}
Consider the interpolation spaces 
$[L^2(I_>;H),H^1_{0,\{0\}}(I_>;H)]_s$ for $0<s<1$, $s\ne \frac12$.
For $0<s<\frac12$, it holds that 
$[L^2(I_>;H),H^1_{0,\{0\}}(I_>;H)]_s = H^s(I_>;H) = [L^2(I_>;H),H^1(I_>;H)]_s$,
i.e., the homogeneous boundary condition at $\{ 0 \}$
is ``not seen'' by the interpolation space, 
whereas for $ \frac12 < s < 1$ we have that 
$$
[L^2(I_>;H),H^1_{0,\{0\}}(I_>;H)]_s = 
 H^s_{0,\{0\}}(I_>;H)  \subset  [L^2(I_>;H),H^1(I_>;H)]_s=H^s(I_>;H)
$$
is a subspace which is norm-closed in $H^s(I_>;H)$, 
\cite[Chap.~1, Remarque~11.3]{LM1}.

\noindent {\upshape (2)} The space $H^{\frac12}_{00,\{0\}}(I_>;H)$, 
which will be important in the present paper, 
is strictly included in 
$H^{\frac12}(I_>;H) = [L^2(I_>;H),H^1(I_>;H)]_{\frac12}$ 
with a topology which is strictly 
finer than that of $ H^{\frac12}(I_>;H)$, \cite[Chap.~1, Thm.~11.7]{LM1}.

\noindent {\upshape (3)} 
$H^{\frac12}_{00,\{0\}}(I_>;H)$ is not closed in the norm of $H^{\frac12}(I_>;H)$.
It is a dense subspace 
(\cite[Theorem~1.4.2.4]{Grisvard85} with $p=2$) 
and the embedding $H^{\frac12}_{00,\{0\}}(I_>;H) \subset H^{\frac12}(I_>;H)$
is continuous, \cite[Lemma~4.8]{Fon09}.

\end{remark}
The following intrinsic characterizations of the spaces of order
$\frac12$ will be useful.  We refer to \cite[Chap.~1]{LM1}, in
particular, for the first one Th\'eor\`eme~9.1 and (10.23) in Section
10.3, and for the second one, Th\'eor\`eme~11.7 and Remarque~11.4.
\begin{proposition}\label{prop:Interp00}
Let $I=(-T,T)$, $I_>=(0,T)$ for $T\in(0,\infty]$.  

\noindent {\upshape (1)} 
The interpolation space $H^{\frac12}(I_>;H)$ consists of all $u\in
L^2(I_>;H)$ which are equal to the restriction to $I_>$ of some
$\tilde{u}\in H^{\frac12}(I;H)$.  The interpolation norm of
$H^{\frac12}(I_>;H)$ is equivalent to the intrinsic norm
$\|\cdot\|_{H^{\frac12}(I_>;H)}$ given by
\be\label{eq:H12norm}
\| u \|^2_{H^{\frac12}(I_>;H)}
=
\| u \|^2_{L^2(I_>;H)}
+
\int_{I_>} \int_{I_>} \frac{\|u(s) - u(t)\|_H^2}{|s-t|^2} \,\dd s\,\dd t
\;.
\ee

\noindent {\upshape (2)}
The interpolation space $H^{\frac12}_{00,\{0\}}(I_>;H)$ consists of
all $u\in H^{\frac12}(I_>;H)$ such that the function $s\mapsto
s^{-\frac12}u(s)$ belongs to $L^2(I_>;H)$ with intrinsic norm
$\|\cdot\|_{H^{\frac12}_{00,\{0\}}(I_>;H)}$ given by
\be\label{eq:H1200norm}
\| u \|^2_{H^{\frac12}_{00,\{0\}}(I_>;H)}
=
\| u \|^2_{L^2(I_>;H)} 
+ 
\int_{I_>} \int_{I_>} \frac{\|u(s) - u(t)\|_H^2}{|s-t|^2} \,\dd s\,\dd t
+
\int_{I_>} \frac{1}{s} \|u(s)\|_H^2 \,\dd s
\;.
\ee
The constants implied by the norm equivalences are independent of 
$T\in(0,\infty]$.
\end{proposition}
\subsection{Fractional calculus on the half line}
\label{sec:FracDeriv}
To render our presentation self-contained, we recapitulate here
fractional calculus from \cite{SamkoEtAl93} as necessary by our
subsequent analysis.

For $\phi\in L^1(\IR_>;\IC)$, $\alpha\in(0,1)$, the Riemann-Liouville
fractional integrals \cite[Def.~2.1]{SamkoEtAl93} are
\begin{align*}
  (\II_{+}^\alpha \phi)(t)
  &=\frac{1}{\Gamma(\alpha)}\int_0^t (t-s)^{\alpha-1}\phi(s)\, \dd s\;,
  \quad t\in\IR_>\;, \\
  (\II_{-}^\alpha \phi)(t)
  &=\frac{1}{\Gamma(\alpha)}\int_t^\infty (s-t)^{\alpha-1}\phi(s)\, \dd s\;,
  \quad t\in\IR_>\;. 
\end{align*} 
Then we have integration by parts \cite[(2.20) and Corollary to Theorem~3.5
p.~67]{SamkoEtAl93}:
\begin{align}
  \label{eq:2}
  \int_{\IR_>} (\II_{+}^\alpha \psi)(t) \phi(t)\,\dd t
  = \int_{\IR_>} \psi(t) (\II_{-}^\alpha \phi)(t)\,\dd t
\end{align}
and the semigroup property \cite[(2.21)]{SamkoEtAl93}: 
\begin{align}
  \label{eq:1}
  \II_{+}^{\alpha+\beta} \phi= \II_{+}^{\alpha} \II_{+}^{\beta} \phi\;, \quad
  \II_{-}^{\alpha+\beta} \phi= \II_{-}^{\alpha} \II_{-}^{\beta} \phi\;, \quad
\alpha,\beta>0\;.
\end{align}
The proofs of \eqref{eq:2}, \eqref{eq:1} are elementary calculations
with integrals.

By $\rD$ we denote the time derivative of order $1$ and we define time
derivatives of fractional order $\alpha \in(0,1)$ for $u\in
C_0^\infty(\IR)$,
\begin{align}
\label{eq:Da+}
(\rD^\alpha_+ u)(t) 
&:= 
(\rD \II_{+}^{1-\alpha}u)(t)
=\frac{1}{\Gamma(1-\alpha)} \rD \int_{0}^t (t-s)^{-\alpha} u(s) \,\dd s\;,
\\
\label{eq:Da-}
(\rD^\alpha_- u)(t) 
&:= 
-(\rD \II_{-}^{1-\alpha}u)(t)
=- \frac{1}{\Gamma(1-\alpha)} \rD \int_{t}^\infty (s-t)^{-\alpha} u(s) \,\dd s\;.
\end{align}
We require a space of test functions, which is closed under the action
of $\rD^{\alpha}_+$ and $\rD^{\alpha}_-$; to this
end we introduce (cp.\ \cite{Fon09})
\begin{align*}
\cF(\IR;\IC) 
:= 
\left\{
u\in C^\infty(\IR;\IC) : \| u \|_{H^s(\IR;\IC)} 
< \infty \;\;\forall s\in \IR
\right\}\;.
\end{align*}
The set $\cF(\IR;\IC)$ is a Fr\'echet space with respect to the
topology induced by the family of norms $\{\| \cdot
\|_{H^s(\IR;\IC)}\}_{s\in\IR}$ and we have the dense embeddings $
\cD(\IR;\IC) \subset \cS(\IR;\IC) \subset \cF(\IR;\IC) \subset
\cE(\IR;\IC)$ (where $\cD$, $\cS$, and $\cE$ are the classical test
function spaces).  We observe that the definitions \eqref{eq:Da+},
\eqref{eq:Da-} remain meaningful for $u\in \cF(\IR;\IC)$.  We further
define test function spaces
\begin{align*}
\cF(\IR_>;\IC) = \left\{u\in C^\infty(\IR_>;\IC): \exists  
  \tilde{u}\in\cF(\IR;\IC) \text{ such that } u=\tilde{u}|_{\IR_>}  \right\}
\end{align*} 
and, with $E_0$ the ``extension by zero'' operator,
\begin{align*}
\cF_0(\IR_>;\IC) 
= \left\{u\in C^\infty(\IR_>;\IC): E_0u\in\cF(\IR;\IC)\right\}\;.
\end{align*}
The subspaces 
$\cF_0(\IR_>;\IC) \subset H^{\frac12}_{00,\{0\}}(\IR_>;\IC)$,
$\cF(\IR_>;\IC) \subset H^{\frac12}(\IR_>;\IC)$  
are dense, 
see \cite[Lemma~3.7]{Fon09}. 

We denote the corresponding spaces of distributions by
$\cF_0'(\IR_>;\IC)=\cF(\IR_>;\IC)^*$ and
$\cF'(\IR_>;\IC)=\cF_0(\IR_>;\IC)^*$.
Then it follows that, see \cite[(2.24)--(2.29)]{Fon09}, 
\begin{align*}
  &\rD_{+}^\alpha\colon \cF_0(\IR_>;\IC)\to \cF_0(\IR_>;\IC)\;, \quad 
  \rD_{-}^\alpha\colon \cF(\IR_>;\IC)\to \cF(\IR_>;\IC)\;,  \\
  &\rD_{+}^\alpha\colon \cF_0'(\IR_>;\IC)\to \cF_0'(\IR_>;\IC)\;, \quad 
  \rD_{-}^\alpha\colon \cF'(\IR_>;\IC)\to \cF'(\IR_>;\IC)\;. 
\end{align*}
Here the $\cF_0'(\IR_>;\IC)$ distribution derivative $\rD_{+}^\alpha$
means 
$$
  \langle \rD_{+}^\alpha u,\phi \rangle 
  =\int_{\IR_>} u \,\rD_{-}^\alpha\phi \,\dd t 
\quad\forall \phi\in \cF(\IR_>;\IC) \;,
$$
and the $\cF'(\IR_>;\IC)$ distribution derivative $\rD_{-}^\alpha$
means 
\begin{align*}
  \langle \rD_{-}^\alpha u,\phi \rangle 
  =\int_{\IR_>} u \,\rD_{+}^\alpha\phi \,\dd t 
\quad\forall \phi\in \cF_0(\IR_>;\IC) \;.
\end{align*}

We can now prove a relevant integration by parts formula.  

\begin{lemma}\label{lem:intbyparts}
The $\cF_0'(\IR_>;\IC)$ distribution derivative $\rD w$ of 
$w\in H^{\frac12}_{00,\{0\}}(\IR_>;\IC)$ 
satisfies 
\begin{align}
  \label{eq:6b} 
  \langle \rD w , v \rangle 
  =  \int_{\IR_>} \rD_{+}^{\frac12}w \,\rD_{-}^{\frac12}v \,\dd t 
\quad\forall v\in \cF(\IR_>;\IC) \;.  
\end{align}
\end{lemma}
\begin{proof} By definition we have
\begin{align*}
      \langle \rD w , v \rangle  
  =  \int_{\IR_>} w \,(-\rD v) \,\dd t 
  =  ( w ,-\rD v) 
\quad\forall v\in \cF(\IR_>;\IC) \;.  
\end{align*} 
Since $\cF_0(\IR_>;\IC) \subset H^{\frac12}_{00,\{0\}}(\IR_>;\IC)$ is dense, it
suffices to show 
\begin{align*}
  (w ,-\rD v)
  =  ( \rD_{+}^{\frac12}w ,\rD_{-}^{\frac12}v )
\quad\forall w\in \cF_0(\IR_>;\IC) \;, \; v\in \cF(\IR_>;\IC) \;.  
\end{align*}
If $w\in \cF_0(\IR_>;\IC)$, then $w(0)=0$, so that $w=\II_{+}^1
\psi=\II_{+}^{\frac12}\II_{+}^{\frac12}\psi$, where $\psi=\rD w$.  Similarly, 
If $v\in \cF(\IR_>;\IC)$, then $v(\infty)=0$, so that $v=\II_{-}^1
\phi=\II_{-}^{\frac12}\II_{-}^{\frac12}\phi$, where $\phi=-\rD v$.
Therefore,  by integration by parts \eqref{eq:2}, 
\begin{align*}
(w ,-\rD v)
= (\II_{+}^{\frac12}\II_{+}^{\frac12}\psi , \phi)
= (\II_{+}^{\frac12}\psi ,\II_{-}^{\frac12}\phi)
= (\rD_{+}^{\frac12}w ,\rD_{-}^{\frac12}v)\;.
\end{align*}
\end{proof}
%
%%%%%%%%%%%%%%%%%%%%%%%%%%%%%%%%%%%%%%%%%%%%%%%%%%%%%%%%%%%%%%%%%%%
\subsection{Extension by zero}
\label{sec:Ext}
%%%%%%%%%%%%%%%%%%%%%%%%%%%%%%%%%%%%%%%%%%%%%%%%%%%%%%%%%%%%%%%%%%%
In the previous Subsection~\ref{sec:Interp} the space
$H^{\frac12}_{00,\{0\}}(\IR_>;H)$ is characterized by means of an
intrinsic norm. Here we give an alternative characterization in terms
of extension by zero, denoted $E_0$.  
\begin{proposition}\label{prop:ZerExt}
The space  $H^{\frac12}_{00,\{0\}}(\IR_>;H)$ is equals $\{ w\in
H^{\frac12}(\IR_>;H): E_0w \in H^{\frac12}(\IR;H) \}$. 
A norm on $H^{\frac12}_{00,\{0\}}(\IR_>;H)$, 
which is equivalent to \eqref{eq:H1200norm},
is given by $\| E_0\cdot \|_{H^{\frac12}(\IR;H)}$; more precisely, 
for every $w\in H^{\frac12}_{00,\{0\}}(\IR_>;H)$ there holds
\be\label{eq:ZeroExCont}
\| w \|^2_{H^{\frac12}_{00,\{0\}}(\IR_>;H)}
\leq 
\| E_0w \|^2_{H^{\frac12}(\IR;H)} 
\leq 
2
\| w \|^2_{H^{\frac12}_{00,\{0\}}(\IR_>;H)}
\;.
\ee
\end{proposition}
\begin{proof}
  For a proof of the first part, we refer to \cite[Lemma~3.5]{Fon09}.
  It remains to show \eqref{eq:ZeroExCont}.  Consider an arbitrary
  $w\in H^{\frac12}_{00,\{0\}}(\IR_>;H)$.  Then $\tilde{w} = E_0w\in
  H^{\frac12}(\IR;H)$ and $\| w \|_{L^2(\IR_>;H)} =\| \tilde{w}
    \|_{L^2(\IR;H)} =\|E_0 w \|_{L^2(\IR;H)}$. We next compute with the
    seminorm defined in \eqref{eq:H12+Seminorm} below: 
\begin{align*}
\begin{split}
| E_0w |^2_{H^{\frac12}(\IR;H)} 
&=| \tilde w |^2_{H^{\frac12}(\IR;H)} 
:= 
\int_{-\infty}^\infty \int_{-\infty}^\infty 
\frac{\| \tilde{w}(t) - \tilde{w}(t') \|_H^2 }{|t-t'|^2}\, \dd t\,\dd t'
\\
& = 
\int_{0}^\infty \int_{0}^\infty \frac{\| \tilde{w}(t) - \tilde{w}(t')
  \|_H^2 }{|t-t'|^2}\, \dd t\, \dd t'
+
2 \int_{t'= -\infty}^0 \int_{t=0}^\infty \frac{\| \tilde{w}(t) -
  \tilde{w}(t') \|_H^2 }{|t-t'|^2} \,\dd t\, \dd t'
\\
& = 
\int_{0}^\infty \int_{0}^\infty \frac{\| {w}(t) - {w}(t')
  \|_H^2 }{|t-t'|^2}\, \dd t\, \dd t'
+
2 \int_{0}^{\infty} \frac{\| w(t) \|_H^2}{t} \,\dd t 
\;.
\end{split}
\end{align*}
By comparison with the seminorm part of the norm
  \eqref{eq:H1200norm}, we conclude that
\begin{align*}
| w |^2_{H_{00,\{0\}}^{\frac12}(\IR_>;H)} 
\le | E_0w |^2_{H^{\frac12}(\IR;H)} 
\le 2 | w |^2_{H_{00,\{0\}}^{\frac12}(\IR_>;H)}  
\end{align*}
and the proof is complete.
\end{proof}
\subsection{Further characterizations}
\label{sec:Characterizations}
The preceding function spaces are intimately connected to the
fractional derivatives $\rD^{\frac12}_+$ and $\rD^{\frac12}_-$ on
$\cF(\IR_>;\IC)$.  As these derivatives are essential in the proposed
space-time formulation, we discuss their properties in detail.  By
continuity, the operators $\rD^{\frac12}_\pm$ extend to bounded
operators from $H^{\frac12}(\IR_>;\IC)$ to $L^2(\IR_>;\IC)$.  The
following proposition collects several properties of
$\rD^{\frac12}_\pm$.
\begin{proposition}\label{prop:Dpm}
Let $H$ denote an arbitrary Hilbert space over $\IR$. Then there holds

\noindent
{\upshape (1)} 
A function $u\in L^2(\IR_>;H)$ belongs to $H^{\frac12}_{00,\{0\}}(\IR_>;H)$
if and only if its $\cF'_0(\IR_>;H)$-derivative  
$\rD^{\frac12}_+ u\in L^2(\IR_>;H) $. 

\noindent 
{\upshape (2)}  
A function $u\in L^2(\IR_>;H)$ belongs to $H^{\frac12}(\IR_>;H)$
if and only if its $\cF'(\IR_>;H)$-derivative
$\rD^{\frac12}_- u \in L^2(\IR_>;H)$.

\noindent
{\upshape (3)}  
A norm on $H^{\frac12}_{00,\{0\}}(\IR_>;H)$, equivalent to 
the norm \eqref{eq:H1200norm}, is given by
\be \label{eq:H12D+norm}
\| u \|^2_{H^{\frac12}_+(\IR_>;H)}
:=
\| u \|^2_{L^2(\IR_>;H)}
+
\| \rD^{\frac12}_+ u \|^2_{L^2(\IR_>;H)}
\;.
\ee
A norm on $H^{\frac12}(\IR_>;H) $, equivalent to the norm
\eqref{eq:H12norm}, is given by
\be\label{eq:H12D-norm}
\| u \|^2_{H^{\frac12}_-(\IR_>;H)}
:=
\| u \|^2_{L^2(\IR_>;H)}
+
\| \rD^{\frac12}_- u \|^2_{L^2(\IR_>;H)}
\;.
\ee
Moreover, for every $u\in H^{\frac12}(\IR;H)$ there holds
\begin{align}
\label{eq:H12pmnorm}
\| u \|^2_{H^{\frac12}(\IR;H)}
&=
\| u \|^2_{L^2(\IR;H)}
+
\| \rD^{\frac12}_- u \|^2_{L^2(\IR;H)}
=
\| u \|^2_{L^2(\IR;H)}
+
\| \rD^{\frac12}_+ u \|^2_{L^2(\IR;H)}
\;,
\\
\label{eq:H12+Seminorm}
\| \rD^{\frac12}_+ u \|^2_{L^2(\IR_;H)}
&\simeq 
| u |^2_{H^{\frac12}(\IR;H)}
:=
\int_{s\in\IR} \int_{t\in\IR} 
\frac{\| u(s) - u(t) \|_H^2}{|s-t|^2} \,\dd s\,\dd t
\;.
\end{align}
\end{proposition}
For the proof of \eqref{eq:H12D+norm}, \eqref{eq:H12D-norm} 
we refer to \cite[Lemmas~3.5, 3.8]{Fon09} 
and to \cite[Lemmas~3.6, 3.9]{Fon09}, respectively.
The identity \eqref{eq:H12pmnorm} is immediate from 
the Fourier characterizations of $\rD^{\frac12}_\pm$ in \cite[Sect.~3]{F00}.
For \eqref{eq:H12+Seminorm}, we refer to \cite[(4.13)]{F00}.
We remark that the expression $|\cdot|_{H^{\frac12}(\IR;H)}$
introduced in \eqref{eq:H12+Seminorm} is indeed a seminorm,
as it vanishes on all functions $u\in H$ independent of $t$.

In view of Lemma~\ref{lem:intbyparts} and
Proposition~\ref{prop:Dpm},
it is now clear that the bilinear form
$\langle \rD w,v\rangle$ is bounded on
$H^{\frac12}_{00,\{0\}}(\IR_>;H)\times H^{\frac12}(\IR_>;H)$.
\subsection{Coercivity over $\IR$}
\label{sec:TimeCoerc}
A key ingredient in the theory of Fontes is that the time derivative
is coercive in the sense of Corollary~\ref{coro:X=Y} for functions
defined on $\IR$. We demonstrate this here by considering the
operator, with $A$ as in Subsection~\ref{sec:FuncSpc},
\begin{align*}
Bv = \rD v + A v \;,\quad v\in \cF(\IR;V)\;.
\end{align*}
By fractional integration by parts (immediate from the Fourier
characterizations of $\rD_\pm^{\frac12}$), we find
\be\label{eq:Tuv}
\langle Bw,v \rangle
=
\int_{\IR} \Big(
(\rD^{\frac12}_+ w,  \rD^{\frac12}_- v)_H
+ 
a(w, v)\Big)
\,\dd t
\;,\quad 
w,v \in \cF(\IR;V)
\;.
\ee
We also define the operator 
\begin{align*}
\HH^\alpha := \cos(\pi\alpha)\mathrm{I} + \sin(\pi\alpha) \HH
\;,\quad \alpha\in\IR\;,
\end{align*}
where $\HH$ is the Hilbert transform acting with respect to the
$t$-variable.  By using \eqref{eq:AssA}, we then obtain the
fundamental \emph{coercivity inequality}: for any $w\in \cF(\IR;V)$
\begin{equation*}
\begin{split}
\langle Bw,\HH^{-\alpha} w \rangle
& = 
\langle \rD w +A w, \cos(\pi\alpha)w-\sin(\pi\alpha)\HH w\rangle
\\ 
& =
\cos(\pi\alpha)
\langle \rD w,  w \rangle 
-\sin(\pi\alpha)
\langle \rD^{\frac12}_+ w, \rD^{\frac12}_-\HH w \rangle 
\\ &  \quad 
+\int_{\IR} \big(\cos(\pi\alpha) a(w,w)
-\sin(\pi\alpha) a(w,\HH w)\big) \,\dd t
\\ &
\ge 
\sin(\pi\alpha)
\|\rD^{\frac12}_+ w\|^2_{L^2(\IR;H)}
+
\big( \lambda_-\cos(\pi\alpha) 
- \lambda_+ \sin(\pi\alpha)\big)\| w \|^2_{L^2(\IR;V)}
\;,
\end{split}
\end{equation*}
because $\langle \rD w, w \rangle =0$, $\|\HH w\|_{L^2(\IR;V)}\le
\|w\|_{L^2(\IR;V)}$, and
$\rD^{\frac12}_-\HH=-\rD^{\frac12}_-\HH^{-\frac12}=-\rD^{\frac12}_+$,
see \cite{Fon09}.  Fixing the parameter $\alpha >0$ sufficiently
small, by density of $\cF(\IR;V)$ in $H^{\frac12}(\IR;H)\cap
L^2(\IR;V)$, and \eqref{eq:H12pmnorm}, we find the coercivity
inequality (cp.\ Corollary~\ref{coro:X=Y}): there exists $c>0$ such
that
\begin{align*}
\langle B w,\HH^{-\alpha} w \rangle
\ge c  \big( \|w\|^2_{H^{\frac12}(\IR;H)}+\|w\|^2_{L^2(\IR;V)}\big)
\quad \forall w\in  H^{\frac12}(\IR;H) \cap L^2(\IR;V)
\;.
\end{align*} 
Hence, by Corollary~\ref{coro:X=Y} and
Proposition~\ref{prop:infsup} we conclude that the bilinear form in
\eqref{eq:Tuv} satisfies the inf-sup conditions \eqref{eq:infsup},
\eqref{eq:injec} with $X=Y=H^{\frac12}(\IR;H) \cap L^2(\IR;V)$. 
%%%%%%%%%%%%%%%%%%%%%%%%%%%%%%%%%%%%%%%%%%%%%%%%%%%%%%%%%
\subsection{Coercivity over $\IR_>$}
\label{sec:coercivityR+}
In order to prove the inf-sup condition \eqref{eq:infsup} for
functions on $\IR_>$, we take an arbitrary $w\in
H^{\frac12}_{00,\{0\}}(\IR_>;H)\cap L^2(\IR_>;V)$.  Then its extension
by zero, $\tilde{w}=E_0w$, belongs to $H^{\frac12}(\IR;H) \cap
L^2(\IR;V)$ according to Proposition~\ref{prop:ZerExt}. Similarly, if
$\tilde{v}\in H^{\frac12}(\IR;H) \cap L^2(\IR;V)$, then its
restriction to $\IR_>$, $v=R_>\tilde{v}$, belongs to
$H^{\frac12}(\IR_>;H) \cap L^2(\IR_>;V)$ according to
Proposition~\ref{prop:Interp00} (1).  We have the bounds 
\begin{align} \label{eq:a11}
  \|w\|_{H_{00,\{0\}}^{\frac12}(\IR_>;H)\cap L^2(\IR_>;V)}
  &\le \|E_0w\|_{H^{\frac12}(\IR;H)\cap L^2(\IR;V)} \;, \\
\label{eq:a12}  
  \|R_>\tilde{v}\|_{H^{\frac12}(\IR_>;H)\cap L^2(\IR_>;V)}
  &\le \|\tilde{v}\|_{H^{\frac12}(\IR;H)\cap L^2(\IR;V)} \;.
\end{align}
Moreover,
\begin{align*}
  (\rD^{\frac12}_+\tilde{w})(t) 
  =\frac{1}{\Gamma(\frac12)} \rD 
  \int_{-\infty}^t (t-s)^{-\frac12}\tilde{w}(s)\,\dd s 
 =
 \begin{cases}
   \frac{1}{\Gamma(\frac12)} \rD 
   \int_{0}^t (t-s)^{-\frac12}w(s)\,\dd s\;,&t>0\;,\\
 0\;,&t<0\;,
\end{cases}
\end{align*}
that is, $\rD^{\frac12}_+E_0w=E_0 \rD^{\frac12}_+w$.  Similarly, 
\begin{align*}
  (\rD^{\frac12}_-\tilde{v})(t) 
  =-\frac{1}{\Gamma(\frac12)} \rD 
  \int_{t}^\infty (s-t)^{-\frac12}\tilde{v}(s)\,\dd s 
  =-\frac{1}{\Gamma(\frac12)} \rD 
  \int_{t}^\infty (s-t)^{-\frac12}v(s)\,\dd s \;, \quad t>0\;,
\end{align*}
that is,
$R_>\rD^{\frac12}_-\tilde{v}=\rD^{\frac12}_-R_>\tilde{v}$. Hence, 
\begin{align*}
\int_{\IR}
(\rD^{\frac12}_+ E_0w,  \rD^{\frac12}_- \tilde{v})_H 
\,\dd t 
&=
\int_{\IR}
(E_0\rD^{\frac12}_+ w,  \rD^{\frac12}_- \tilde{v})_H 
\,\dd t 
=
\int_{\IR_>}
(\rD^{\frac12}_+ w,  R_>\rD^{\frac12}_- \tilde{v})_H 
\,\dd t 
\\ &=
\int_{\IR_>}
(\rD^{\frac12}_+ w,  \rD^{\frac12}_- R_>\tilde{v})_H 
\,\dd t \;.
\end{align*}
If we denote by $\cB_{\IR}(\cdot,\cdot)$ and
$\cB_{\IR_>}(\cdot,\cdot)$ bilinear forms as in \eqref{eq:Tuv}
computed over $\IR$ and $\IR_>$, respectively, then we conclude that
\begin{align} \label{eq:a13}
  \cB_{\IR_>}(w,R_>\tilde{v})
=\cB_{\IR}(E_0w,\tilde{v})\;. 
\end{align}

The inf-sup condition proved in the previous subsection means that for
each $\tilde{w}\in H^{\frac12}(\IR;H) \cap L^2(\IR;V)$ there is a
$\tilde{v}\in H^{\frac12}(\IR;H) \cap L^2(\IR;V)$ (namely,
$\tilde{v}=\HH^{-\alpha}\tilde{w}$) such that
\begin{align}\label{eq:a14}
  \frac{\cB_{\IR}(\tilde{w},\tilde{v})}
       {\|\tilde{v}\|_{H^{\frac12}(\IR;H)\cap L^2(\IR;V)}} 
   \ge c {\|\tilde{w}\|_{H^{\frac12}(\IR;H)\cap L^2(\IR;V)}}\;. 
\end{align}
For arbitrary $w\in H^{\frac12}_{00,\{0\}}(\IR_>;H)\cap L^2(\IR_>;V)$,
we let $\tilde{w}=E_0w$ and take $\tilde{v}$ as above and set
$v=R_>\tilde{v}$, that is, $v=R_>\HH^{-\alpha}E_0w$.  Then, by 
\eqref{eq:a11},  
\eqref{eq:a12},  
\eqref{eq:a13}, and   
\eqref{eq:a14}, we obtain  
\begin{align*}
  \frac{\cB_{\IR_>}(w,v)}
       {\|v\|_{H^{\frac12}(\IR_>;H)\cap L^2(\IR_>;V)}} 
&  \ge
  \frac{\cB_{\IR}(\tilde{w},\tilde{v})}
       {\|\tilde{v}\|_{H^{\frac12}(\IR;H)\cap L^2(\IR;V)}} 
\\ &
   \ge c \|\tilde{w}\|_{H^{\frac12}(\IR;H)\cap L^2(\IR;V)}
   \ge c \|w\|_{H_{00,\{0\}}^{\frac12}(\IR_>;H)\cap L^2(\IR_>;V)}\;.
\end{align*}
This is the desired inf-sup condition. 
%%%%%%%%%%%%%%%%%%%%%%%%%%%%%%%%%%%%%%%%%%%%%%%%%%%%%%%%%
\section{Linear parabolic evolution equations}
\label{sec:PDEs}
%%%%%%%%%%%%%%%%%%%%%%%%%%%%%%%%%%%%%%%%%%%%%%%%%%%%%%%%%
%
We present a space-time variational formulation of the initial
boundary value problem for the abstract, linear parabolic evolution
equation \eqref{eq:ParProb} with homogeneous initial condition
\eqref{eq:ParIC}.  For the operator $A\in \cL(V,V^*)$, we assume
\eqref{eq:AssA}.  In what follows, all Hilbert spaces are taken over
the coefficient field $\IR$.
Using the function spaces developed in Section~\ref{sec:Prels}, we now
state the weak form of the linear parabolic initial-value problem
\eqref{eq:ParProb}, \eqref{eq:ParIC}: it is based on the Bochner
spaces
\be\label{eq:ChoicXY}
\begin{aligned}
X &= H^{\frac12}_{00,\{0\}}(\IR_>;H)\cap L^2(\IR_>;V)
   \simeq 
    \big(H^{\frac12}_{00,\{0\}}(\IR_>)\otimes H\big) 
    \cap \big(L^2(\IR_>)\otimes V \big)
\;, 
\\
Y &= H^{\frac12}(\IR_>;H)\cap L^2(\IR_>;V)
  \simeq \big(H^{\frac12}(\IR_>)\otimes H\big) 
    \cap \big( L^2(\IR_>) \otimes V \big)
\;.
\end{aligned}
\ee
Here, $\otimes$ signifies the Hilbert tensor product space endowed
with the (unique) cross norm.  The parabolic operator takes the form $B
= \rD + A$ with the $\cF'_0$-distributional derivative $\rD$
introduced in Section~\ref{sec:FracDeriv}, Lemma~\ref{lem:intbyparts}.

Besides the spaces $X$ and $Y$ in \eqref{eq:ChoicXY}, we will also
need the space
\be\label{eq:defZ}
Z = H^{\frac12}(\IR;H) \cap L^2(\IR;V) \;.
\ee
We shall make use of the following continuity properties of extensions
and restrictions which follow from Proposition~\ref{prop:Interp00} and
Proposition~\ref{prop:ZerExt}.
\begin{proposition}\label{prop:ExtRest}
  For $X$, $Y$, and $Z$ as in \eqref{eq:ChoicXY}, \eqref{eq:defZ} there
  holds:

\noindent {\upshape(1)} 
$X\subset Z$ with continuous embedding given
by the zero extension $E_0$.

\noindent {\upshape(2)}
$Y = R_>(Z)$ with 
$R_>$ denoting the operator of restriction
of elements of $L^2(\IR;H)$ to $\IR_>$.

\noindent {\upshape(3)}
$Z^* \simeq (H^{\frac12}(\IR;H))^* + L^2(\IR;V)^* \simeq 
H^{-\frac12}(\IR;H) + L^2(\IR;V^*)$.
 
\noindent {\upshape(4)}
$Y^*$ is isomorphic to $\{ g\in Z^*: {\rm supp}(g) \subseteq \IR_> \}$.

\noindent {\upshape(5)} 
$X$ is a dense subset of $Y$, that is, 
$\overline{X}^{\| \cdot \|_Y} = Y$.
\end{proposition}
From $A\in \cL(V,V^*)$ it follows that $B:=\rD+A\in \cL(X,Y^*)$.
More precisely, there holds for every $v\in X$,
\begin{align*}
B v &= (\rD+A)v =  \rD v + Av
\in  (H^{\frac12}(\IR_>;H))^* + L^2(\IR_>;V^*) 
\\ &\simeq  (H^{\frac12}(\IR_>;H))^* + L^2(\IR_>;V)^*
\simeq  (H^{\frac12}(\IR_>;H)\cap L^2(\IR_>;V))^*
=  Y^*  \;.
\end{align*}
For any source term $f\in Y^*$, we consider the {\em space-time weak
  formulation} of \eqref{eq:ParProb}, \eqref{eq:ParIC}: find
\be\label{eq:ParaWeak}
u\in X:\quad 
\cB_{\rD+A}(u,v) = F(v) \quad \forall v\in Y
\;.
\ee
Here, the linear functional $F(\cdot)$ is defined by 
\begin{align*}
F(v) = \langle f, v\rangle \quad  \forall v\in  Y 
\end{align*}
with $\langle\cdot,\cdot\rangle$ denoting the $Y^* \times Y $
duality pairing. The bilinear form is given by,
cp.~Lemma~\ref{lem:intbyparts},
\be\label{eq:DefB}
\cB_{\rD+A}(w,v)
:=
\int_{\IR_>} 
\Big\{ (\rD^{\frac12}_+ w , \rD^{\frac12}_- v)_H + a(w,v) \Big\}
\,\dd t\;, 
\quad 
w\in X,\ v \in Y\;,
\ee
where $X$ and $Y$ are as in \eqref{eq:ChoicXY}.  The form
$\cB_{\rD+A}(\cdot,\cdot)$ in \eqref{eq:DefB} is continuous by
Proposition~\ref{prop:Dpm} (1) and (2), stating that for every $w\in
H^{\frac12}_{00,\{0\}}(\IR_>;H)$ we have $\rD^{\frac12}_+w \in
L^2(\IR_>;H)$ and that for every $v\in H^{\frac12}(\IR_>;H)$ we have
$\rD^{\frac12}_-v \in L^2(\IR_>;H)$.

The unique solvability of \eqref{eq:ParaWeak} was proved in
\cite[Sect.~4.1]{Fon09} by extension to a problem over $\IR$, where
coercivity in the sense of Corollary~\ref{coro:X=Y} can be proved, see
Subsection~\ref{sec:TimeCoerc}. As a result of the unique solvability
of \eqref{eq:ParaWeak} we conclude that the inf-sup conditions
\eqref{eq:infsup}, \eqref{eq:injec} hold.  We formulate this in the
following proposition.
\begin{proposition}\label{prop:xtinfsup} 
Suppose that assumption \eqref{eq:AssA} holds. Then,
for the choice \eqref{eq:ChoicXY} of spaces, 
the bilinear form \eqref{eq:DefB} 
satisfies the continuity condition \eqref{eq:BilCont}
and the inf-sup conditions \eqref{eq:infsup}, \eqref{eq:injec}.
In particular, for every $f\in Y^*$ there exists 
a unique solution $u\in X$ of \eqref{eq:ParaWeak}.
\end{proposition}

\begin{proof}
  We observe that $Y \simeq B^{1,\frac12}_{0,\cdot}(Q_+)$ and that $X
  \simeq B^{1,\frac12}_{0,0}(Q_+)$ in the notation of \cite[Thm.~4.3,
  Sect.~4.1]{Fon09} with $p=2$.  It is shown there that the operator
  $B = \rD + A \in \cL(X,Y^*)$ is bijective. Therefore,
  Proposition~\ref{prop:infsup} implies the inf-sup conditions
  \eqref{eq:infsup}, \eqref{eq:injec} for the bilinear form
  \eqref{eq:DefB} on the spaces $X\times Y$ in \eqref{eq:ChoicXY}.
\end{proof}
\section{Sparse tensor Galerkin discretization}
\label{sec:WavGalDisc}
Having established the well-posedness and the unique solvability of
\eqref{eq:ParProb}, \eqref{eq:ParIC} we now turn to Galerkin
approximations. Rather than considering time-stepping (as studied,
e.g., in \cite{Thomee}), we are interested in {\em compressive
  space-time Galerkin} discretizations, as analyzed for the first time
in \cite{ScSt09}.  We present and analyze adaptive, compressive,
space-time schemes which are based on the weak space-time formulation
\eqref{eq:ParaWeak}.  The adaptive, and space-time compressive schemes
inherit, being instances of the general theory in \cite{CDD1,CDD2},
stability from the well-posedness of the infinite-dimensional problem
shown in Proposition~\ref{prop:xtinfsup} and from the stability of the
Riesz bases. As in \cite{ScSt09}, they are based on tensor product
constructions of Riesz bases of $X$ and $Y$; however, the variational
formulation \eqref{eq:ParaWeak} obviates the need for stability of
Riesz bases in negative order Sobolev spaces.
We present classes of spline wavelets in the time domain and also in the
spatial domain $D\subset \IR^{n}$, which we assume to be a polygon or
polyhedron.  Rather than focusing on a particular family of wavelets,
we specify several axioms from \cite{ScSt09} to be satisfied by the
tensorized multiresolution bases in the spatial and temporal domains
in order for our analysis to apply.
We assume that $V$ and $H$ are modeled on Sobolev spaces on the bounded Lipschitz
polyhedron $D\subset \IR^n$, $n\geq 1$. 
As in \cite{ScSt09}, our analysis accommodates two cases:
case \eqref{A}: $n=2,3$ and $D$ is a
bounded polyhedron with plane faces; and the high-dimensional case
\eqref{B}: $n\geq 1$ and $D=(0,1)^n$.  In $D$ we consider
general elliptic operators $A$ of order $2m$, $m \ge 1$.  The generic
example is $A = -\Delta$, $V=H^1_0(D)$, and $H=L^2(D)$, in which case
$m=1$.  The domain for the parabolic initial-boundary value problem is
the space-time cylinder $Q_> := \IR_>\times D$.
\subsection{Space-time wavelet Galerkin discretization}
\label{sec:xtGal}
The Galerkin discretization of the space-time variational formulation
\eqref{eq:ParaWeak} will be based on two dense, nested families
$\{X^\ell\}_{\ell\in\IN_0}$, $\{Y^\ell\}_{\ell\in\IN_0}$ of subspaces
of $X$ and $Y$ as in \eqref{eq:ChoicXY}.  The inf-sup condition
\eqref{eq:infsup} makes it necessary to allow $X^\ell \ne Y^\ell$
(leading in effect to Petrov-Galerkin discretizations), so that
Proposition~\ref{prop:infsup} is used in full generality.  As
indicated above, we choose $\{X^\ell\}_{\ell\in\IN_0}$ as
tensor-products of spaces of continuous, piecewise polynomial
functions of $t\in \IR_>$ and $x\in D$, in order to obtain good
(space-time compressive) approximation of solutions, whereas $Y^\ell$
will be selected to ensure good stability.  Multiresolution bases will
be required to ensure: (a) multilevel preconditioning, i.e., all
stiffness matrices have (generalized) condition numbers, which are
bounded independently of $\ell$; and (b) matrix and (space-time)
solution compression.

Thus, we consider the Galerkin discretization: 
to find, for $\ell \in \IN_0$,
\be\label{eq:ParOpGal}
u^\ell \in X^\ell:\quad 
\cB_{\rD + A}(u^\ell,v^\ell) = F(v^\ell) \quad \forall v^\ell \in Y^\ell 
\;.
\ee
We assume that 
\begin{align*}
N_\ell = {\rm dim}(X^\ell) = {\rm dim}(Y^\ell) <\infty \;,
\end{align*}
such that 
$X^\ell \subset X=H^{\frac12}_{00,\{0\}}(\IR_>;H)\cap
L^2(\IR_>;V)$ and $Y^\ell \subset Y=H^{\frac12}(\IR_>;H)\cap L^2(\IR_>;V)$
are closed and $\cup_{\ell\in\IN} X^\ell$ and 
$\cup_{\ell\in\IN} Y^\ell$ are dense in $X$, respectively in $Y$.
Proposition~\ref{prop:infsup} implies
\begin{proposition}\label{prop:QuasiOpt}
Assume that the Galerkin discretization \eqref{eq:ParOpGal}
of \eqref{eq:ParaWeak} is stable, in the sense that there
exists $\bar{\gamma}$ such that, for all $\ell\in\IN_0$,
\be\label{eq:infsupell}
\inf_{0\ne w\in X^\ell} \sup_{0\ne v \in Y^\ell}
\frac{\cB_{\rD+A}(w,v)}{\| w \|_X \| v\|_Y}  \geq \bar{\gamma} > 0
\;.
\ee
Then, for every $F\in Y^*$ and for every $\ell\in \IN$,
the Galerkin approximation \eqref{eq:ParOpGal} 
admits a unique solution $u^\ell\in X^\ell$.
In particular, the (in general, non-symmetric) stiffness matrix
corresponding to \eqref{eq:ParOpGal} is nonsingular. 
Let $u\in X$ be the corresponding unique solution to
\eqref{eq:ParaWeak} and $C$ be the constant in
\eqref{eq:BilCont}. 
Then there holds the quasi-optimality estimate
\be\label{eq:QuasiOpt}
\| u - u^\ell \|_X 
\leq 
\frac{C}{\bar{\gamma}} 
\inf_{v^\ell \in X^\ell} \| u - v^\ell \|_X
\;.
\ee
\end{proposition}
The proof of Proposition~\ref{prop:QuasiOpt} is straightforward:
existence and uniqueness of $u^\ell$ in \eqref{eq:ParOpGal} and the
invertibility of the $N_\ell\times N_\ell$ matrix follows from
\eqref{eq:infsupell} 
with Proposition~\ref{prop:infsup}.  The error estimate
\eqref{eq:QuasiOpt} follows from the Galerkin orthogonality
$$
\cB_{\rD+A}(u-u^\ell,v^\ell) = 0 \qquad \forall v^\ell \in Y^\ell 
\;, 
$$
by noting that the error is $u-u^\ell=(I-R^\ell)(u-v^\ell)$, where
$R^\ell$ is the Ritz projector that maps $u\mapsto u^\ell$.
Therefore, \eqref{eq:QuasiOpt} 
holds with constant
$\|I-R^\ell\|_{\cL(X,X)} =\|R^\ell\|_{\cL(X,X)}\le C/\bar{\gamma}$,
\cite{XuZikatanov03}.

For preconditioning and efficient computation, as well for 
{\em adaptive space-time Galerkin discretizations} with 
optimality properties, the concept of {\em Riesz basis} 
takes a central role.
\subsection{Riesz bases and bi-infinite matrix vector equations}
\label{sec:RieszMVProb}
We assume at hand a Riesz basis
$\Psi^{X}=\{\psi_{\lambda}^{X}:\lambda \in \nabla^{X} \}$ for $X$.
The Riesz basis property amounts to saying that the {\em synthesis operator}
$$
s_{\Psi^{X}}:\ell_2(\nabla^{X}) \rightarrow X:{\bf c} \mapsto {\bf c}^\top 
   \Psi^{X} := \sum_{\lambda \in \nabla^{X}} c_{\lambda} \psi_{\lambda}^{X}
$$
is boundedly invertible. 
Its adjoint, known as the {\em analysis operator}, reads
$$
s_{\Psi^{X}}':X^* \rightarrow \ell_2(\nabla^{X}):
g\mapsto [g(\psi_{\lambda}^{X})]_{\lambda \in \nabla^{X}}
\;.
$$
Similarly, let 
$\Psi^{Y} = \{\psi_{\lambda}^{Y}:\lambda \in \nabla^{Y}\}$ 
denote a Riesz basis for $Y$, with synthesis operator $s_{\Psi^{Y}}$ 
and adjoint $s_{\Psi^{Y}}'$. 
Ahead, we construct Riesz bases $\Psi^{X}$ and $\Psi^{Y}$ by 
tensorization of wavelet bases in $\IR_>$ and in $D$.

By Proposition~\ref{prop:xtinfsup},
$B=\rD+A \in \cL(X,Y^*)$ is boundedly invertible
with the choice of spaces in \eqref{eq:ChoicXY}.
We may write \eqref{eq:ParaWeak}
equivalently as operator equation:
given $f \in Y^*$, 
find
\be\label{eq:OpEq}
u\in X:\quad B u  = f \;\;\mbox{in}\;\; Y^*
\;.
\ee
Writing $u=s_{\Psi^{X}} {\bf u}$, 
\eqref{eq:ParaWeak} and \eqref{eq:OpEq} are
equivalent to the bi-infinite matrix vector problem
\begin{equation} \label{mv}
{\bf B} {\bf u}={\bf f}\;,
\end{equation}
where 
${\bf f}
=
s'_{\Psi^{Y}} f 
=[f(\psi_{\lambda}^{Y})]_{\lambda \in \nabla^{Y}} \in \ell_2(\nabla^{Y})$\;, 
and where the {\em ``stiffness''} or {\em system matrix}
$$
\matr{B} =
s'_{\Psi^{Y}} B s_{\Psi^{X}} 
= [(B \psi_{\mu}^{X})(\psi_{\lambda}^{Y})]_{\lambda \in \nabla^{Y}, 
\mu \in \nabla^{X}}  \in \cL(\ell_2(\nabla^{X}),\ell_2(\nabla^{Y}))
$$
is boundedly invertible.  
We may write
$$
\cB_{\rD + A}: 
X \times Y \rightarrow \IR:(w,v) \mapsto (B w)(v)\;,
$$
and we also use the notations
$$
{\bf B} = \cB_{\rD+A}(\Psi^{X},\Psi^{Y}) 
\quad \mbox{and} \quad {\bf f}=f(\Psi^{Y})
\;.
$$
With the Riesz constants
\begin{align*}
\Lambda_{\Psi^{X}}^{X}
:=
\|s_{\Psi^{X}}\|_{\cL(\ell_2(\nabla^{X}), X)} 
&=\sup_{0\neq {\bf c} \in \ell_2(\nabla^{X})} 
\frac{\|{\bf c}^\top \Psi^{X}\|_{X}}{\|{\bf c}\|_{\ell_2(\nabla^{X})}}\;,
\\
\lambda_{\Psi^{X}}^{X}
:=
\|s^{-1}_{\Psi^{X}}\|_{\cL(X,\ell_2(\nabla^{X}))}^{-1}
& =\inf_{0\neq {\bf c} \in \ell_2(\nabla^{X})} 
\frac{\|{\bf c}^\top \Psi^{X}\|_{X}}{\|{\bf c}\|_{\ell_2(\nabla^{X})}}\;,
\end{align*}
and analogous constants 
$\Lambda_{\Psi^{Y}}^{Y}$ and $\lambda_{\Psi^{Y}}^{Y}$,
the bounded invertibility of $B \in \cL(X,Y^*)$ implies that
the condition number of $\matr{B}$ is finite, i.e., 
\begin{align*} 
\|{\bf B}\|_{\cL(\ell_2(\nabla^{X}), \ell_2(\nabla^{Y}))} 
& \leq  \|B\|_{\cL(X, Y^*)} \Lambda_{\Psi^{X}}^{X} \Lambda_{\Psi^{Y}}^{Y}\;,
\\
\|{\bf B}^{-1}\|_{\cL(\ell_2(\nabla^{Y}), \ell_2(\nabla^{X}))} 
& \leq 
\frac{\|B^{-1}\|_{\cL(Y^* , X)}}{\lambda_{\Psi^{X}}^{X} \lambda_{\Psi^{Y}}^{Y}} 
\;.
\end{align*}
We next construct {\em Riesz bases} of the spaces $X$ and $Y$ in
\eqref{eq:ChoicXY}.
%%%%%%%%%%%%%%%%%%%%%%%%%%%%%%%%%%%%%%%%%%%%%%%%%%%%%%%%%%%%%%%%%%%%%%%%%%%%%%%
\subsection{Riesz bases in $H^{\frac12}_{00,\{0\}}(\IR_>)$ and $H^{\frac12}(\IR_>)$}
\label{sec:WavR+}
%%%%%%%%%%%%%%%%%%%%%%%%%%%%%%%%%%%%%%%%%%%%%%%%%%%%%%%%%%%%%%%%%%%%%%%%%%%%%%%
We assume at our disposal
two countable collections $\Theta^X, \Theta^Y \subset H^1(\IR_>)$ of
functions such that
$$
\Theta^X=\{\theta^X_{\lambda}: \lambda \in \nabla^X_t\} \subset H^1_{0,\{0\}}(\IR_>)
$$
is a 
{\em normalized Riesz basis for $L^2(\IR_>)$ which, when 
renormalized in $H^1(\IR_>)$, is a Riesz basis for  $H^1_{0,\{0\}}(\IR_>)$.}
Analogously, we assume available
$\Theta^Y = \{\theta^Y_{\lambda}: \lambda \in \nabla^Y_t\} \subset H^1(\IR_>)$, 
a Riesz basis of $L^2(\IR_>)$ which,
when renormalized in $H^1(\IR_>)$, is a Riesz basis for $H^1(\IR_>)$.

From Proposition~\ref{prop:Interp00} 
we obtain the following result.
\begin{proposition}\label{prop:RieszTheta}
Assume given two collections $\Theta^X$ and $\Theta^Y$ with the 
above properties. 
Then, for $0\leq s \leq 1$, the collections
$[\Theta^X]_s$ and $[\Theta^Y]_s$, which are obtained by 
rescaling $\Theta^X$ and $\Theta^Y$ by 
$\{ 2^{s|\lambda|}: \lambda\in \nabla_t\}$,
(e.g., $[\Theta^X]_s = \{ 2^{s|\lambda|} \theta^X_\lambda: \lambda \in \nabla_t^X \}$)
are Riesz bases of $[L^2(\IR_>),H^1_{0,\{0\}}(\IR_>)]_s$ 
and of $[L^2(\IR_>),H^1(\IR_>)]_s$, respectively.
In particular, for $s=\frac12$, 
$[\Theta^X]_{\frac12}$ is a Riesz basis for $H^{\frac12}_{00,\{0\}}(\IR_>)$
and
$[\Theta^Y]_{\frac12}$ is a Riesz basis for $H^{\frac12}(\IR_>)$.
\end{proposition}
We denote by $\theta^X_\lambda$ elements of the collection 
$\Theta^X$ and, likewise, by $\theta^Y_\lambda$ elements of $\Theta^Y$.
Further assumptions on the bases 
$\Theta^X$, $\Theta^Y$ are as in \cite{ScSt09}: 
denoting by $\theta_{\lambda}$ a generic element in either
of the collections $\Theta^X$ and $\Theta^Y$, we require
the $\theta_\lambda$ to be 
\begin{enumerate} 
\renewcommand{\theenumi}{t\arabic{enumi}}
\renewcommand{\labelenumi}{(t\arabic{enumi})}
\item \label{t1}
{\em local}: that is, 
$\sup_{t \in \IR_>, \ell \in \N_0} 
  \#\{\lambda : |\lambda|=\ell,\ t \in {\rm supp}\, \theta_{\lambda}\} <\infty$ 
and 
$|{\rm supp}\,\theta_{\lambda}| \lesssim 2^{-|\lambda|}$,
\item \label{t2}
{\em piecewise polynomial of order $d_t$}: 
here, ``piecewise'' means that
the singular support consists of a finite number of points 
whose number is uniformly bounded with respect to $|\lambda|$,
\item \label{t3}
{\em globally  continuous}: specifically, 
$\|\theta_{\lambda}\|_{W_{\infty}^k(\IR_>)} 
 \lesssim  2^{|\lambda|(\frac{1}{2}+k)}$ for $k \in \{0,1\}$,
\item \label{t4}
{\em vanishing moments}:
for $|\lambda|>0$, the $\theta_{\lambda}$ have 
$\tilde{d}_t \geq d_t$ vanishing moments.
\end{enumerate}
Properties \eqref{t1}--\eqref{t4} are assumed to hold for both 
$\Theta^X$ and $\Theta^Y$.
We remark that property \eqref{t3}, global continuity, is necessary
to ensure $H^{\frac12}(\IR_>)$-conformity, even though $H^{\frac12}(\IR_>)$
is not embedded into $C^0(\overline{\IR_>})$.

Properties \eqref{t1}--\eqref{t4} can be satisfied by collections
$\Theta^X$, $\Theta^Y$ that are continuous, piecewise polynomial
wavelet bases on dyadic refinements of $\IR_>$, which are of {\em order
  $d_t>1$}.  For $k \in \N_0$ we denote by $\nabla_t^{(k)}$ the set of
$\lambda \in \nabla_t$ with {\em refinement level} $|\lambda| \leq k$.
It holds that $\# \nabla_t^{(k)} \eqsim 2^k$.  Setting also
$\nabla_t^{(-1)}:=\emptyset$, we define the biorthogonal projector
$Q^X_{k,t} := Q^X_{\nabla_t^{(k)}}$ by
\begin{align*}
  Q^X_{k,t}v=\sum_{\lambda\in \nabla_t^{(k)X}} \langle
  v,\theta_\lambda'^X\rangle \theta_\lambda^X \;, 
\end{align*}
where $\Theta'^{X}$ denotes the dual basis, and analogously for
$Q^Y_{k,t} := Q^Y_{\nabla_t^{(k)}}$. 
We have
\begin{align*}
\|\mathrm{Id}-Q^X_{k,t}\|_{\cL(H^{d_t}(\IR_>), L^{2}(\IR_>))}
\lesssim 2^{-k d_t}\;,
\quad 
\|\mathrm{Id}-Q^X_{k,t}\|_{\cL(H^{d_t}(\IR_>), H^{\frac12}_{00,\{0\}}(\IR_>))} 
\lesssim 2^{-k (d_t - \frac12)} 
\end{align*}
and analogously for $ Q^Y_{k,t}$ with 
$H^{\frac12}(\IR_>)$ in place of $H^{\frac12}_{00,\{0\}}(\IR_>)$.
 
{\em Constructions} of 
compactly supported spline wavelet 
systems $\Theta$ on $(-1,1)$ and on $\IR$, 
as well as {\em direct constructions} 
(i.e., not based on antisymmetry) 
of Riesz bases $\Theta^X$ and $\Theta^Y$ 
on $\IR_>$ satisfying properties \eqref{t1}--\eqref{t4}  
with $\theta^X_\lambda(t)|_{t=0} = 0$ 
are available, for example,
in \cite{ChegSt11,DKU99,DijCSSt09,Urb09}  
and the references there.
\subsection{Riesz bases in $H$ and $V$}
\label{sec:WavD}
With $H=L^2(D)$ and the assumption 
that $V$ coincides with a closed subspace (supporting homogeneous
essential boundary conditions) of the Sobolev space $H^m(D)$ 
for some $m > 0$, we assume at our disposal a Riesz basis
$$
\Sigma=\{\sigma_{\lambda}: \lambda \in \nabla_x\} \subset V
\;.
$$
Specifically, $\Sigma$ is a collection of functions that is a 
{\em normalized Riesz basis for $H$ which, upon renormalization 
in $V$, is a Riesz basis denoted $[\Sigma]_V$ for $V$}. 
Riesz bases of divergence-free functions in the context 
of Example~\ref{ex:Expl2} are constructed in 
\cite{StevDiv0,Urb09} and the references there.
For the spatial wavelet basis $\Sigma$, we 
consider as in \cite{ScSt09}, two cases: 
\begin{enumerate}
\renewcommand{\theenumi}{\Alph{enumi}}
\renewcommand{\labelenumi}{(\Alph{enumi})}
\item \label{A} 
it is a wavelet basis of order $d_x>m$ 
with isotropic supports
constructed from a dyadic multiresolution analysis in 
$L^2(D)$, 
\item \label{B}
$D = (0,1)^n$ and $\Sigma$ is the tensor product 
of (possibly different) 
univariate wavelet bases $\Sigma_i$ as in \eqref{A} 
in each of the coordinate spaces.
\end{enumerate}
In case \eqref{A}, for some sufficiently large $K$ depending on $m$,
where $2m$ is the order of $A$, and for some 
$r_x\in\N_0 $ such that  $m-1 \leq r_x \leq d_x-2$ and $\tilde{d}_x \in \N_0$,
we will assume that the $\sigma_{\lambda}$ are 
\begin{enumerate} \renewcommand{\theenumi}{s\arabic{enumi}}
\renewcommand{\labelenumi}{(s\arabic{enumi})}
\item \label{s1} {\em local} and {\em piecewise smooth}: 
for any $\ell \in \N_0$ there exist collections 
$\{D_{\ell,v}: v \in {\mathcal O}_{\ell}\}$ of disjoint, 
uniformly shape regular, open subdomains such that 
$\overline{D}
=
\cup_{v \in {\mathcal O}_{\ell}} \overline{D_{\ell,v}}$, 
$\overline{D_{\ell,v}}$ is the union of some $\overline{D_{\ell+1,\tilde{v}}}$,
${\rm diam}(D_{\ell,v}) \eqsim 2^{-\ell}$, 
${\rm supp}\, \sigma_{\lambda}$ is connected and is the union of 
a uniformly bounded number of 
$\overline{D_{|\lambda|,v}}$, each $\overline{D_{\ell,v}}$ 
has non-empty intersection with the supports of a uniformly bounded number 
of $\sigma_{\lambda}$ with $|\lambda|=\ell$, and, for $k \in \{0,K\}$, 
$$
\|\sigma_{\lambda}\|_{W^k_{\infty}(D_{|\lambda|,v}
)} \lesssim 2^{|\lambda|(\frac{n}{2}+k)}\;,
$$
\item \label{s2} 
{\em globally $C^{r_x}$}: 
specifically, $\|\sigma_{\lambda}\|_{W_{\infty}^k(D)} 
\lesssim  
2^{|\lambda|(\frac{n}{2}+k)}$ for $k \in \{0,r_x+1\}$,
\item \label{s3} 
for $|\lambda|>0$, have {\em cancellation properties of order $\tilde{d}_x$}:
\begin{equation*} 
\Big|\int_{D} w \sigma_{\lambda}\Big| 
\lesssim 
2^{-|\lambda|(\frac{n}{2}+k)} \|w\|_{W_{\infty}^k(D)}\  
\text{ for }k\in \{0,\tilde{d}_x\},\,w \in W_{\infty}^k(D) \cap V
\;.
\end{equation*}
\item \label{s4} In addition to \eqref{s1}, 
we assume that for any $\ell$ and $v \in {\mathcal O}_{\ell}$, 
there exists a sufficiently smooth transformation of coordinates $\kappa$, 
with derivatives bounded uniformly in $\ell$ and $v$, 
such that for all $|\lambda|=\ell$, 
$(\sigma_{\lambda} \circ \kappa)|_{\kappa^{-1}(D_{\ell,v})}$ 
is a polynomial of some fixed degree.
\end{enumerate}

For case \eqref{B}, we assume that each of the $\Sigma_i$ satisfies
the above conditions with $(D,n)=((0,1),1)$.  In this case, we assume
that the wavelets are {\em piecewise polynomials of order $d_x$}, with
those on positive levels being orthogonal to all polynomials of order
$\tilde{d}_x$ that are in $V$. 
\begin{assumption} \label{asmp:s*compress}
The bi-infinite matrices
$\matr{M} = (\Sigma,\Sigma)_H$ and $\matr{A} = a([\Sigma]_V,[\Sigma]_V)$
for the spatial operators in \eqref{eq:GalMat} are
$s^*$ computable, in the sense that for each $N\in \IN$, there
exist approximate matrices $\matr{M}_N$ and $\matr{A}_N$
with at most $N$ non-zero entries in each column and such that,
for every $0\leq \bar{s} < s^*$, the expressions
$$
\sup_{N\in \IN} N \| \matr{M} - \matr{M}_N \|^{1/\bar{s}}\;, 
\quad 
\sup_{N\in \IN} N \| \matr{A} - \matr{A}_N \|^{1/\bar{s}}
$$
are finite.  Here, $\| \cdot \|$ denotes the spectral norm.
\end{assumption}
A number of practically viable constructions of Riesz bases $\Sigma$,
which satisfy Assumption~\ref{asmp:s*compress} for several classes of
operators $A\in \cL(V,V^*)$ have become available in recent years: for
example, for second order, elliptic divergence form differential
operators $A$, and also for self-adjoint, integro-differential
operators $A$ of fractional order (in which case $V$ coincides with
the domain of $A^{\frac12}$); also tensorized $\Sigma$ for diffusions on
$D=(0,1)^{n}$ have become available, which satisfy
Assumption~\ref{asmp:s*compress}.  We refer to \cite[Sect.~8.3]{ScSt09}
for this.  For $0\leq s\leq 1$, we denote by $[\Sigma]_{s}$ the Riesz
basis $\Sigma$ rescaled to $[H,V]_{s}$.
\subsection{Riesz bases in $X$ and $Y$}
\label{sec:RieszInXY}
We assume that we have at our disposal Riesz bases
$\Theta^X = \{\theta^X_\lambda : \lambda \in \nabla^X_t \}$, $\Theta^Y
= \{\theta^Y_\lambda : \lambda \in \nabla^Y_t \}$ of $L^2(\IR_>)$ for
which rescaling renders $\Theta^X$ a Riesz basis of
$H^1_{0,\{0\}}(\IR_>)$ and $\Theta^Y$ a Riesz basis of $H^1(\IR_>)$.  The
bases $[\Theta^X]_{\frac12}$ and $[\Theta^Y]_{\frac12}$ are then defined as in
Proposition~\ref{prop:RieszTheta}.  In the spatial domain $D$, we
assume available a Riesz basis $\Sigma = \{ \sigma_\lambda : \lambda
\in \nabla_x \}$ of $H$ which, when rescaled to $V$, becomes a Riesz
basis $[\Sigma]_V$ for $V$: $[\Sigma]_V = \{ \sigma_\lambda / \|
\sigma_\lambda \|_V : \lambda \in \nabla_x \}$.
\begin{proposition}\label{prop:RieszInX=Y}
Given Riesz bases
$\Theta^X$, $\Theta^Y$ and $\Sigma$ of 
$L^{2}(\IR_>)$ and $H$,
respectively, as above, the collections
$\Psi^X := \Theta^X\otimes \Sigma$ ,
$\Psi^Y := \Theta^Y\otimes \Sigma$ ,
are Riesz bases of 
$L^2(\IR_>;H) \simeq L^2(\IR_>)\otimes H$.
Moreover, the collection
\begin{align*}
\Psi^X
:= 
\left\{ 
(t,x)\mapsto 
\frac{\theta^X_\lambda(t)\sigma_\mu(x)}
{\sqrt{\| \sigma_\mu \|_V^2 + \| \theta^X_\lambda \|^2_{H^{\frac12}_{00,\{0\}}(\IR_>)}
}
}
\;
:
(\lambda,\mu)\in \nabla^X := \nabla_t^X \times \nabla_x 
\right\}
\end{align*}
is a Riesz basis for 
$X=H^{\frac12}_{00,\{0\}}(\IR_>;H)\cap L^2(\IR_>;V)$,
and the collection
\begin{align*}
\Psi^Y
:= 
\left\{ 
(t,x)\mapsto 
\frac{\theta^Y_\lambda(t)\sigma_\mu(x)}
{\sqrt{\| \sigma_\mu \|_V^2 + \| \theta^Y_\lambda \|^2_{H^{\frac12}(\IR_>)}
}
}
\;
:
(\lambda,\mu)\in \nabla^Y := \nabla_t^Y \times \nabla_x 
\right\}
\end{align*}
is a Riesz basis for $Y=H^{\frac12}(\IR_>;H)\cap L^2(\IR_>;V)$.

The Riesz constants for $\Psi^X$ and $\Psi^Y$
depend only on the respective Riesz constants 
for 
$\Theta^X$, $[\Theta^X]_{\frac12}$,
$\Theta^Y$, $[\Theta^Y]_{\frac12}$ 
and for $\Sigma$, $[\Sigma]_V$.
\end{proposition}
\begin{proof} 
The Riesz basis property for $\Psi$
follows from our assumptions on $\Theta$ and $\Sigma$,
the result \cite[Prop.~1, Prop.~2]{GO95} on tensor products of Riesz bases
and from Proposition~\ref{prop:RieszTheta}.
\end{proof}

\subsection{Space-time compressible approximation rates of smooth solutions}
\label{sec:XYchoice}
Using the tensor product Riesz bases $\Psi^X$ of $X$ in
Proposition~\ref{prop:RieszInX=Y} in a Petrov-Galerkin discretization
\eqref{eq:ParOpGal} of the space-time variational formulation
\eqref{eq:ParaWeak} allows for space-time compressive approximations
of smooth solutions, {\em provided} test function spaces $Y^\ell$ are
available which are {\em stable}, i.e., which satisfy
\eqref{eq:infsupell}.
The approximate solutions thus obtained will be quasi-optimal.  Such
stable test spaces can be constructed on the basis of the coercivity
property in Subsection~\ref{sec:coercivityR+}. However, we shall not
develop this here but refer to \cite[Chapt.~5]{FDiss}.  Likewise, in
the adaptive setting, sequences of approximate solutions are produced,
which converge at best possible rates, when compared to best $N$-term
approximations of the solution.  We therefore exemplify the best
possible approximation rates in $X$ which can be achieved
%and the so-called ``mix-regularity'' (i.e., simultaneous
%increase of smoothness with respect to the spatial and the temporal variables)
in terms of the parameters $d_t$ and $d_x$.
%%%%%%%%%%%%%%%%%%%%%%%%%%%%%%%%%%%%%%%%%%%%%%%5
\subsubsection{Best rate in case \eqref{A}} 
%%%%%%%%%%%%%%%%%%%%%%%%%%%%%%%%%%%%%%%%%%%%%%%5
For any $\Lambda \subset \nabla_x$, let $Q_{\Lambda}:L^2(D) \to
\spann(\theta_\lambda:\lambda\in\Lambda)$ denote the
$L^2(D)$-biorthogonal projector associated to $\Sigma$ and $\Lambda$.
The assumption of $\Sigma$ being of order $d_x$ means that, with
$\nabla_x^{(k)}$ being the set of $\lambda \in \nabla_x$ with
refinement level $|\lambda| \leq k \in \N_0$, it holds that $\#
\nabla_x^{(k)} \eqsim 2^{k n}$.  Setting $\nabla_x^{(-1)}:=\emptyset$,
we obtain for the projector $Q_{k,x}:=Q_{\nabla_x^{(k)}}$ that
$$
\|\mathrm{Id}-Q_{k,x}\|_{\cL(H^{d_x}(D) \cap V , V)} \lesssim 2^{-k (d_x-m)},\,
\;\;
\|\mathrm{Id}-Q_{k,x}\|_{\cL(H^{d_x}(D) \cap V , H)} \lesssim 2^{-k d_x}
\;.
$$
In case $d_t <\frac{d_x-m}{n}$,
with
$\ell/k \in [\frac{d_t}{d_x-m}+\eps,\frac{1}{n}-\eps]$ 
for (small) $\eps > 0$, 
we have
$$
\Big\|
\mathrm{Id} - \sum_{p=0}^k \sum_{q=0}^{\ell} 
(Q_{p,t}-Q_{p-1,t}) \otimes (Q_{q,x}-Q_{q-1,x})
\Big\|_{\cL(H^{d_t}(\IR_>) \otimes (H^{d_x}(D) \cap V), 
L^2(\IR_>)\otimes V)} \lesssim 2^{-k d_t}
\;.
$$
Here
$\sum_{p=0}^k \sum_{q=0}^{\ell} (Q_{p,t}-Q_{p-1,t}) \otimes (Q_{q,x}-Q_{q-1,x})$ 
is the  $L^2(D)$-biorthogonal projector associated to the 
tensor product basis $\Psi=\Theta \otimes \Sigma$ and the 
``sparse'' tensor-product index set 
\begin{align*}
\Lambda_{\mathrm{A}}
:= 
\cup_{p=0}^k \cup_{q=0}^{\ell} 
(\nabla_t^{(p)} \backslash \nabla_t^{(p-1)}) 
\times (\nabla_x^{(q)} \backslash \nabla_x^{(q-1)})\;,
\end{align*}
which satisfies $\#(\Lambda_{\mathrm{A}}) \lesssim 2^k$, see \cite{GH13}. 

In view of the approximation orders of the bases being applied, 
and the tensor product structure of 
$X = H^{\frac12}_{00,\{0\}}(\IR_>;H) \cap L^2(\IR_>;V)$, 
by interpolation we obtain the rate
\be\label{eq:rateSpTt}
2^{-k[\min(d_t-\frac12,{\textstyle \frac{d_x-m}{n}})- \eps]}
\;. 
\ee
with $\eps>0$ arbitrarily small due to the appearance of logarithmic
factors.  This rate is best possible for functions which are smooth
with respect to $x$ and $t$, and for Riesz bases $\Sigma$ with
isotropic supports in $D$ as are admitted in case \eqref{A}.
%%%%%%%%%%%%%%%%%%%%%%%%%%%%%%%%%%%%%%%%%%%%%%%%%%%%%%%%%%%%%%%%%%%%%%%
\subsubsection{Best rate in case \eqref{B}} 
%%%%%%%%%%%%%%%%%%%%%%%%%%%%%%%%%%%%%%%%%%%%%%%%%%%%%%%%%%%%%%%%%%%%%%
Throughout the discussion of case \eqref{B}, we assume  $n\geq 2$
(the case $n=1$ being a particular instance of \eqref{A}).
For $1 \leq i \leq n$, let $V_i$ be either $H^m(0,1)$ 
or a closed subspace incorporating essential boundary conditions. 
Let $\Sigma_i=\{\sigma_{i,\lambda_i}:\lambda_i \in \nabla_i\}$ be a
normalized Riesz basis for $H_i:=L^2(0,1)$, that renormalized in 
$V_i$ is a Riesz basis for $V_i$.
For any $\Lambda_i \subset \nabla_i$ we denote by 
$Q_{\Lambda_i}:L^2(0,1) \rightarrow \spann(\theta_{\lambda}:\lambda\in\nabla_i)$
the $L^2(0,1)$-biorthogonal projectors associated to $\Sigma_i$ and $\Lambda_i$. 
The assumption of $\Sigma_i$ consisting of continuous, piecewise polynomial
functions of order $d_x$ means that, with 
$\nabla_i^{(k)} =\{ \lambda \in \nabla_i: |\lambda| \leq k \in \N_0\}$, 
on any finite subinterval $(0,T) \subset \IR_>$
it holds that $\# \nabla_i^{(k)} \eqsim 2^{k}$ 
(with the constant implied in $\eqsim$ being $O(T)$).
With the convention $\nabla_i^{(-1)}:=\emptyset$,
and $Q_{-1,i} \equiv 0$, 
we have for $Q_{k,i}:= Q_{\nabla_i^{(k)}}$ that 
$$
\|\mathrm{Id}-Q_{k,i}\|_{\cL(H^{d_x}(0,1) \cap V_i , V_i)} \lesssim 2^{-k (d_x-m)}\;,
\quad
\|\mathrm{Id}-Q_{k,i}\|_{\cL(H^{d_x}(0,1) \cap V_i, H_i)} \lesssim 2^{-k d_x}
\;.
$$
The collection 
$\Sigma:=\otimes_{i=1}^n \Sigma_i=\{\sigma_{\lambda}:=\otimes_{i=1}^n \sigma_{i,\lambda_i}: 
\lambda \in \nabla_x:=\prod_{i=1}^n \nabla_i\}$ 
is a normalized Riesz basis for $L^2(D)$.
Rescaling this basis in 
$$
V:=\cap_{i=1}^n \otimes_{j=1}^n W_{i j}\;, \mbox{ where } 
W_{i j}:=
\begin{cases}
H_j\;, & \text{when } j \neq i\;,
\\
V_i\;, & \text{when } j =i\;,
\end{cases}
$$
it is a Riesz basis for $V$ as well.

Recall that for any $\Lambda \subset \nabla_x$, 
$Q_{\Lambda}$ denotes the $L^2(D)$-biorthogonal projector associated to 
$\Sigma$ and $\Lambda$.
As shown in \cite{GH13,Todor09}, there exist ``optimized'' sparse product sets 
$\nabla_x^{(1)} \subset \nabla_x^{(1)} \subset \cdots \subset \nabla_x$ 
and 
$\hat{\nabla}_x^{(1)} \subset \hat{\nabla}_x^{(1)} \subset \cdots \subset \nabla_x$ 
with 
$\#\nabla_x^{(1)} \eqsim 2^k \eqsim \hat{\nabla}_x^{(k)}$, 
such that with 
$Q_{k,x} := Q_{\nabla_x^{(k)}}$ 
and with 
$\hat{Q}_{k,x}:=Q_{\hat{\nabla}_x^{(k)}}$, 
and 
$$
\mathcal{H}^{d_x}(D)
:=
\cap_{i=1}^n \otimes_{j=1}^n Z_{i j}\;, \text{ where } 
Z_{i j}:=
\begin{cases}
H_j\;, & \text{when } j \neq i\;,
\\
H^{d_x}(0,1) \cap V_i\;, & \text{when } j =i\;,
\end{cases}
$$
it holds that
$$
\|\mathrm{Id}-Q_{k,x}\|_{\cL(\mathcal{H}^{d_x}(D), V)}
\lesssim 2^{-k (d_x-m)}\;,
\quad
\|\mathrm{Id}-\hat{Q}_{k,x}\|_{\cL(\mathcal{H}^{d_x}(D), H)}
\lesssim 2^{-kd_x}
\;.
$$
Choosing as index set $\Lambda_{\mathrm{B}}$ the 
union of sparse products of the index sets 
$(\nabla_t^{(p)})_{0\leq p \leq k}$ 
with $(\nabla_x^{(q)})_{0\leq q \leq \ell}$ 
or 
$(\hat{\nabla}_x^{(q)})_{0\leq q \leq \ell}$ for suitable $k$ and $\ell$,
we obtain $L^2(\IR_> \times D)$-biorthogonal projectors associated to 
$X_{\Lambda_{\mathrm{B}}} \subset X = \mathrm{clos}_X(\Theta \otimes \Sigma)$ 
that, 
for $u \in (H^{d_t}\cap H^1_{0,\{0\}})(\IR_>) \otimes \mathcal{H}^{d_x}(D)$, 
with a set of at most $N$ basis functions
give rise to an error in $H^s_{0,\{0\}}(\IR_>;V)$ 
of order $2^{-k \min(d_t-s,d_x-m)}$, for $s=0,1$.
Interpolation between $L^2(\IR_>)$ and $H^1_{0,\{0\}}(\IR_>)$ 
results, by Proposition~\ref{prop:Interp00},
in the norm of the Bochner space 
$X \simeq H^{\frac12}_{00,\{0\}}(\IR_>)\otimes V$ 
in the (best possible, for smooth functions) rate
\be\label{eq:bestSpTtx}
\min(d_t-\frac12,d_x-m)  
\;.
\ee
Summarizing \eqref{eq:rateSpTt} and \eqref{eq:bestSpTtx}, for
solutions which are smooth functions of space and time, the rate
%%%%%%%%%%%%%%%%%%%%%%%%%%%%%%%%%%%%%%%%%%%%%%%%%%%%%%%%%%%%%%%%%%%5
\begin{equation} \label{eq:sast}
s_{\mathrm{max}} 
:= 
\begin{cases}
\min(d_t-\frac12,\frac{d_x-m}{n})-\varepsilon & \text{in case } \eqref{A}\;,
\\
\min(d_t-\frac12,d_x-m) & \text{in case } \eqref{B}\;.
\end{cases}
\end{equation}
%%%%%%%%%%%%%%%%%%%%%%%%%%%%%%%%%%%%%%%%%%%%%%%%%%%%%%%%%%%%%%%%%%%%%%
is realized with the index sets $\Lambda_{\mathrm{A}},
\Lambda_{\mathrm{B}}\subset \nabla^X$.
%%%%%%%%%%%%%%%%%%%%%%%%%%%%%%%%%%%%%%%%%%%%%%%%%%%%%%%55
\section{Adaptivity}
\label{sec:Adap}
The sparse tensor space-time Galerkin discretization
\eqref{eq:ParOpGal} based on the a priori choices
$X_{\Lambda_{\mathrm{A}}}$, $X_{\Lambda_{\mathrm{B}}}$ of sparse
tensor product trial spaces and the corresponding testfunction spaces
$Y_{\Lambda_{\mathrm{A}}}$, $Y_{\Lambda_{\mathrm{B}}}$ lead to
quasi-optimal approximations; the quality of the Galerkin
approximation thus being determined by the best approximation
property.  Alternatively, following \cite{CDD2,GHS07}, (sequences of)
subspaces $X^\ell = X_{\Lambda_\ell^X} \subset X$ and $Y^\ell =
Y_{\Lambda_\ell^Y} \subset Y$ may be selected {\em adaptively}, with
sequences $\{ \Lambda_k \}_{k\geq 0}\subset \nabla_t\times \nabla_x$
of sets of ``active'' basis elements $\theta_{\lambda}\otimes
\sigma_{\mu}\in \Psi=\Theta\otimes\Sigma$ determined so as to
ensure optimality properties of the corresponding Galerkin
approximations $u_{\Lambda_k^X}$ {\em for the given set of data}.  In
doing this, a key role is played by the (approximate) computability of
(finite sections of) the bi-infinite matrix $\matr{B}$ defined by
\be\label{eq:GalMat}
\matr{B}
= 
\big( \big(\rD^{\frac12}_+ [\Theta^X]_{\frac12}), 
\rD^{\frac12}_-[\Theta^Y]_{\frac12}\big)\big)_{L^2(\IR_>)}
\otimes
\big(\Sigma,\Sigma\big)_H
+ \big( \Theta^X , \Theta^Y \big)_{L^2(\IR_>)} 
\otimes a\big([\Sigma]_V,[\Sigma]_V\big)
\;.
\ee
We recapitulate basic properties of adaptive wavelet-Galerkin methods,
in particular, the notions of admissibility and computability of the
corresponding discretized operators; our presentation will be
synoptic, and we refer readers who are unfamiliar with these to
\cite{Stev09Surv,ScSt09}.  We will, in particular, review the notions
of $s$-admissibility, $s$-computability and $s$-compressibility of
Galerkin matrices of operators. Finally, we obtain an optimality
result for the adaptive wavelet Galerkin discretization of the
space-time variational formulation \eqref{eq:ParaWeak}: the sequence
of Galerkin solutions produced by the adaptive scheme is optimal in
the norm of $X$ with respect to the best $N$-term approximation of
the solution in space-time tensor product wavelet bases; thereby
offering the first result on optimality for a nonlinear and
compressive algorithm for long-time parabolic evolution problems.
This is distinct from \cite{ChegSt11,ScSt09}, where the constants in
the error and complexity estimates depend on the length of the time
interval.
\subsection{Nonlinear approximation}
\label{sec:NonlAppr}
Nonlinear approximations to $u\in X$ are obtained from
its coefficient vector ${\bf u}$ by best $N$-term approximations ${\bf u}_N$.
These vectors, with supports of size $N \in \N_0$, encode 
the $N$ largest coefficients in modulus of ${\bf u}$. 
For $s>0$, the approximation class 
$
\cA_{\infty}^s(\ell_2(\nabla^{X})):=\big\{{\bf v}\in \ell_2(\nabla^{X}): 
       \|{\bf v}\|_{\cA_{\infty}^s(\ell_2(\nabla^{X}))}<\infty\big\},
$
where
$$
\|{\bf v}\|_{\cA_{\infty}^s(\ell_2(\nabla^{X}))}
:=
\sup_{\delta>0} 
\delta \times 
[\min\{N \in \N_0:\|{\bf v}-{\bf v}_N\|_{\ell_2(\nabla^{X})} 
\leq \delta\}]^s
$$
contains all ${\bf v}$ whose best 
$N$-term approximations converge to ${\bf v}$ with rate $s$. 

Since best $N$-term approximations involve searching the entire vector
${\bf v}$, they cannot be realized in practice.  In addition, for a
solution $u\in X$ of the PDE \eqref{eq:ParProb}, the vector ${\bf u}$
to be approximated is not explicitly available. It is only given
implicitly via \eqref{eq:ParProb}, \eqref{eq:ParIC} through the
(equivalent) bi-infinite matrix vector problem \eqref{mv} with respect
to some Riesz basis $\Psi^X$.  Our aim is to construct a practical
method that produces approximations to ${\bf u}$ which, whenever ${\bf
  u} \in \cA^s_{\infty}(\ell_2(\nabla^{X}))$ for some $s>0$, converge
with this rate $s$ in linear computational complexity.
\subsection{Adaptive Galerkin methods}
\label{sec:AdGal}
Let $s>0$ be such that ${\bf u} \in
\cA^s_{\infty}(\ell_2(\nabla^{X}))$.  In \cite{CDD2} and the
references there, {\em adaptive wavelet Galerkin methods} for solving
\eqref{mv} were introduced.  These methods are iterative methods which
address the non-elliptic nature of the operator \eqref{eq:ParProb} by
iterating, instead of \eqref{eq:GalMat}, on the associated normal
equations, i.e., on the linear system
\begin{equation} \label{eq:normal}
{\bf B}^{\ast} {\bf B} {\bf u} ={\bf B}^{\ast} {\bf f}\;.
\end{equation}
Key ingredients in the estimates of their complexity are asymptotic
cost bounds for approximate matrix-vector products in terms of the
prescribed tolerance $\eps$.
\begin{definition}\label{def:sadmiss} ($s^*$-admissibility) 
  ${\bf B} \in \cL(\ell_2(\nabla^{X}),\ell_2(\nabla^{Y}))$ {\em is
    $s^{\ast}$-admissible if there exists a routine
$$
{\bf APPLY}_{\bf B}[{\bf w},\eps] \rightarrow {\bf z}
$$
which yields, for any $\eps>0$ and any finitely supported 
${\bf w} \in \ell_2(\nabla^{X})$,
a finitely supported ${\bf z} \in \ell_2(\nabla^{Y})$
with $\|{\bf B} {\bf w}-{\bf z}\|_{\ell_2(\nabla^{Y})}\leq \eps$
and for which, for any $\bar{s} \in (0,s^{\ast})$, there exists an 
{\em admissibility constant} $a_{{\bf B},\bar{s}}$ such that 
$\# {\rm supp}\,{\bf z} \leq a_{{\bf B},\bar{s}} \eps^{-1/\bar{s}} 
 \|{\bf w}\|^{1/\bar{s}}_{\cA_{\infty}^{\bar{s}}(\ell_2(\nabla^{X}))}$, 
and the number of arithmetic operations and storage locations used by 
the call ${\bf APPLY}_{\bf B}[{\bf w},\eps]$
is bounded by some absolute multiple of
$$a_{{\bf B},\bar{s}}\eps^{-1/\bar{s}}\|{\bf w}\|_{\cA_{\infty}^{\bar{s}}
          (\ell_2(\nabla^{X}))}^{1/\bar{s}}+ \#{\rm supp}\,{\bf w} + 1.
$$
}
\end{definition}
One key step in adaptive wavelet methods for \eqref{mv}
is thus the construction of a valid routine 
${\bf APPLY}_{\bf B}[{\bf w},\eps]$ 
for the bi-infinite matrices ${\bf B}$ defined in \eqref{eq:GalMat}.

In order to approximate ${\bf u}$ one should be able to approximate ${\bf f}$. 
{\em Throughout what follows, we therefore 
assume availability of the following routine.}

{\em
${\bf RHS}_{\bf f}[\eps] \rightarrow {\bf f}_{\eps}:$
For given $\eps>0$,
it yields a finitely supported
${\bf f}_{\eps} \in \ell_2(\nabla^{Y})$ with 
$$\| {\bf f}-{\bf f}_{\eps} \|_{\ell_2(\nabla^{Y})}\leq \eps 
  \quad\mbox{and}\quad
  \#{\rm supp}\,{\bf f}_{\eps} \lesssim \min\{N:\|{\bf f}-{\bf f}_{N}\|\leq \eps\},
$$ 
with the number of arithmetic operations and storage locations used by
the call ${\bf RHS}_{\bf f}[\eps]$
bounded by some absolute multiple of $\#{\rm supp}\,{\bf f}_{\eps}+1$.
}

The availability of 
${\bf APPLY}_{\bf B}$ and ${\bf RHS}_{\bf f}$ implies the 
following result.
\begin{proposition} \label{prop1}
Let ${\bf B}$ in \eqref{mv} be $s^{\ast}$-admissible. 
Then for any $\bar{s} \in (0,s^{\ast})$, we have
$\|{\bf B}\|_{\cL(\cA^{\bar{s}}_{\infty}(\ell_2(\nabla^{X})), 
  \cA^{\bar{s}}_{\infty}(\ell_2(\nabla^{Y})))}
\leq a_{{\bf B},\bar{s}}^{\bar{s}}$.
For ${\bf z}_{\eps}:={\bf APPLY}_{\bf B}[{\bf w},\eps]$, 
there holds
$\|{\bf z}_{\eps}\|_{\cA_{\infty}^{\bar{s}}(\ell_2(\nabla^{Y}))} 
 \leq 
 a_{{\bf B},\bar{s}}^{\bar{s}} \|{\bf w}\|_{\cA_{\infty}^{\bar{s}}(\ell_2(\nabla^{X}))}$.
\end{proposition}
For proofs, we refer to \cite{CDD2} or \cite[Prop.~3.3]{DFR07}.
Using the definition of $ \cA_{\infty}^s(\ell_2(\nabla^{Y}))$ and the properties of 
${\bf RHS}_{\bf f}$,  we have 
\begin{corollary} \label{coro:fadmiss} If, in \eqref{mv}, ${\bf B}$ is
  $s^*$-admissible and ${\bf u} \in
  \cA_{\infty}^s(\ell_2(\nabla^{X}))$ for $s<s^{\ast}$, then for ${\bf
    f}_{\eps}={\bf RHS}_{\bf f}[\eps]$, $\#{\rm supp}\,{\bf f}_{\eps}
  \lesssim a_{{\bf B},s} \eps^{-1/s} \|{\bf
    u}\|_{\cA_{\infty}^s(\ell_2(\nabla^{X}))}^{1/s}$ with the number
  of arithmetic operations and storage locations used by the call
  ${\bf RHS}_{\bf f}[\eps]$ being bounded by some absolute multiple of
$$
a_{{\bf B},s}  \eps^{-1/s}\|{\bf u}\|_{\cA_{\infty}^s(\ell_2(\nabla^{X}))}^{1/s}+1.
$$
\end{corollary}

\begin{remark} \label{remmie}
Besides $\|{\bf f}-{\bf f}_{\eps}\|_{\ell_2(\nabla^{Y})} \leq \eps$, 
the complexity bounds in Corollary~\ref{coro:fadmiss} with $a_{{\bf B},s}>0$
being independent of $\eps$ 
are essential for the use of ${\bf RHS}_{\bf f}$ in the adaptive wavelet methods.
\end{remark}
The following corollary of Proposition~\ref{prop1} can be used for example for the 
construction of valid {\bf APPLY} and {\bf RHS} routines in case the adaptive wavelet 
algorithms are applied to a preconditioned system.
\begin{corollary} \label{corol}
If
${\bf B} \in \cL(\ell_2(\nabla^{X}),\ell_2(\nabla^{Y}))$, 
${\bf C} \in \cL(\ell_2(\nabla^{Y}),\ell_2(\nabla^{Z}))$ are both $s^{\ast}$-admissible, 
then so is ${\bf C} {\bf B} \in \cL(\ell_2(\nabla^{X}),\ell_2(\nabla^{Z}))$. 
A valid  routine ${\bf APPLY}_{{\bf C} {\bf B}}$ is 
\begin{equation} \label{eq:applyCapplyB}
[{\bf w},\eps] \mapsto 
{\bf APPLY}_{{\bf C}}\big[{\bf APPLY}_{{\bf B}}[{\bf w},\eps/(2\|{\bf C}\|)],\eps/2\big]\;,
\end{equation}
with admissibility constant
$a_{{\bf C} {\bf B},\bar{s}}
  \lesssim a_{{\bf B},\bar{s}}(\|{\bf C}\|^{1/\bar{s}}+a_{{\bf C},\bar{s}})$
for $\bar{s} \in (0,s^{\ast})$.

For some $s^{\ast}>s$, 
let ${\bf C} \in \cL(\ell_2(\nabla^{Y}),\ell_2(\nabla^{Z}))$ be $s^{\ast}$-admissible. 
Then for
\begin{equation} \label{113}
{\bf RHS}_{{\bf C} {\bf f}}[\eps]
:=
{\bf APPLY}_{\bf C}[{\bf RHS}_{{\bf f}}[\eps/(2 \|{\bf C}\|)],\eps/2]\;,
\end{equation}
there holds 
\begin{align*}
& \#{\rm supp}\,{\bf RHS}_{{\bf C} {\bf f}}[\eps]
\lesssim a_{{\bf B},s}(\|{\bf C}\|^{1/s}+a_{{\bf C},s}) \eps^{-1/s}
 \|{\bf u}\|_{\cA_{\infty}^s(\ell_2(\nabla^{X}))}^{1/s}\;,
\\ 
& \|{\bf C} {\bf f}-{\bf RHS}_{{\bf C} {\bf f}}[\eps] \|_{\ell_2(\nabla^{Z})}
\leq \eps\;,
\end{align*}
with the number of arithmetic operations and storage locations used by
the call ${\bf RHS}_{{\bf C} {\bf f}}[\eps]$ bounded by a
multiple of 
$$a_{{\bf B},s}(\|{\bf C}\|^{1/s}+a_{{\bf C},s}) \eps^{-1/s}
\|{\bf u}\|_{\cA_{\infty}^s(\ell_2(\nabla^{X}))}^{1/s}+1
\;.
$$
\end{corollary}
\begin{remark}
${\bf RHS}_{{\bf C} {\bf f}}$ allows to 
approximate ${\bf C}{\bf f}$ in the sense of Remark~\ref{remmie}.
\end{remark}
Consider first the case that ${\bf B}$ is {\em self-adjoint positive
  definite}, i.e., $\nabla^{X}=\nabla^{Y}$ and ${\bf B}={\bf
  B}^{\ast} >0$.  In this case the adaptive wavelet methods from
\cite{CDD2} are {\em optimal} in the following sense.

\begin{theorem} \label{thopt} (\cite{CDD2,GHS07}) If in (\ref{mv})
  ${\bf B}$ is self-adjoint positive definite and $s^*$-admissible,
  then for any $\eps>0$, the adaptive wavelet method from \cite{CDD2}
  produces an approximation ${\bf u}_{\eps}$ to ${\bf u}$ with $\|{\bf
    u}-{\bf u}_{\eps}\|_{\ell_2(\nabla^{X})} \leq \eps$.  If in
  \eqref{mv} for some $s > 0$ it holds ${\bf u} \in
  \cA^s_{\infty}(\ell_2(\nabla^{X}))$, then $\# {\rm supp}\,{\bf
    u}_{\eps} \lesssim \eps^{-1/s}\|{\bf
    u}\|_{\cA^s_{\infty}(\ell_2(\nabla^{X}))}^{1/s}$ and if, in
  addition, $s<s^*$, the number of arithmetic operations and storage
  locations required by one call of either of these adaptive wavelet
  solvers with tolerance $\eps$ is bounded by a multiple of
$$
\eps^{-1/s}(1+a_{{\bf B},s})\|{\bf u}\|_{\cA^s_{\infty}(\ell_2(\nabla^{X}))}^{1/s}+1.
$$
The factor depends only on $s$ when it tends to $0$ or $\infty$,
and on $\|{\bf B}\|$ and  $\|{\bf B}^{-1}\|$.
\end{theorem}

The adaptive Galerkin discretization method from \cite{CDD1} for
self-adjoint operators ${\bf B}$ consists of the application of a
damped Richardson iteration to ${\bf B}{\bf u}={\bf f}$, where the
required residual computations are approximated using calls of ${\bf
  APPLY}_{\bf B}$ and ${\bf RHS}_{\bf f}$ within tolerances that
decrease linearly with the iteration counter.

With the method from \cite{CDD1}, a sequence $\Xi_0 \subset \Xi_1
\subset \cdots\subset \nabla^{X}$ is produced, together with
corresponding (approximate) Galerkin solutions ${\bf u}_i \in
\ell_2(\Xi_i)$.  The coefficients of approximate residuals ${\bf
  f}-{\bf B}{\bf u}_i$ are used as indicators how to expand $\Xi_i$ to
$\Xi_{i+1}$ such that it gives rise to an improved Galerkin
approximation.

The method of \cite{CDD1} relies on a recurrent coarsening of the
approximation vectors, where small coefficients are removed to
maintain optimal balance between accuracy and support length.  We have
$s^{\ast}$-admissibility of ${\bf B}$ once the stiffness matrix with
respect to suitable wavelet bases is close to a computable sparse
matrix.  The next definition makes this precise.
\begin{definition} ($s^{\ast}$-computability) \label{def:scomput} 
${\bf B} \in \cL(\ell_2(\nabla^{X}),\ell_2(\nabla^{Y}))$ 
{\em is $s^{\ast}$-computable if, for each $N \in \N$, there exists a 
${\bf B}_N \in  \cL(\ell_2(\nabla^{X}),\ell_2(\nabla^{Y}))$
having in each column at most $N$ non-zero entries whose joint computation 
takes an absolute multiple of $N$ operations, such that 
the
{\em computability constants}
$$
c_{{\bf B},\bar{s}}:=\sup_{N \in \N} N \|{\bf B}-{\bf
  B}_N\|_{\ell_2(\nabla^{X}) \rightarrow  
\ell_2(\nabla^{Y})}^{1/\bar{s}}
$$
are finite for any $\bar{s}\in (0,s^{\ast})$.
}
\end{definition}

\begin{theorem} \label{th10} 
An $s^{\ast}$-computable ${\bf B}$ is $s^{\ast}$-admissible.
Moreover, for $\bar{s}<s^{\ast}$, $a_{{\bf B},\bar{s}} \lesssim c_{{\bf B},\bar{s}}$ 
where the constant in this estimate depends only on $\bar{s}
\downarrow 0$, $\bar{s} \uparrow s^{\ast}$,  
and on $\|{\bf B}\| \rightarrow \infty$.
\end{theorem}
This theorem is proven by the construction of a suitable ${\bf
  APPLY}_{\bf B}$ routine as was done in \cite[\S6.4]{CDD1}, see also
\cite{Stev09Surv} and the references there.

The non-elliptic nature of ${\bf B}$ was addressed in \cite{CDD2} by
applying the adaptive schemes to the {\em normal equations}
\eqref{eq:normal}: From Subsection~\ref{sec:RieszMVProb} we deduce
that ${\bf B}^{\ast} {\bf B} \in
\cL(\ell_2(\nabla^{X}),\ell_2(\nabla^{X}))$ is boundedly invertible,
self-adjoint positive definite, with
\begin{align*}
\|{\bf B}^{\ast} {\bf B}\|_{\cL(\ell_2(\nabla^{X}), \ell_2(\nabla^{X}))} 
&\leq 
\|{\bf B}\|_{\cL(\ell_2(\nabla^{X}), \ell_2(\nabla^{Y}))}^2
\;,\\ 
\|({\bf B}^{\ast} {\bf B})^{-1}
\|_{\cL(\ell_2(\nabla^{X}), \ell_2(\nabla^{X}))} 
&\leq 
\|{\bf B}^{-1}\|_{\cL(\ell_2(\nabla^{Y}), \ell_2(\nabla^{X}))}^2
\;.
\end{align*}
Now let ${\bf u} \in \cA^s_{\infty}(\ell_2(\nabla^{X}))$, and assume
that for some $s^{\ast}>s$, both ${\bf B}$ and ${\bf B}^{\ast}$ are
$s^{\ast}$-admissible.  By Corollary~\ref{corol}, with ${\bf B}^*$ in
place of ${\bf C}$, a valid ${\bf RHS}_{{\bf B}^{\ast} {\bf f}}$
routine is given by \eqref{113}, and ${\bf B}^{\ast} {\bf B}$ is
$s^{\ast}$-admissible with a valid ${\bf APPLY}_{{\bf B}^{\ast} {\bf
    B}}$ routine given by \eqref{eq:applyCapplyB}.  A combination of
Theorem~\ref{thopt} and Corollary~\ref{corol} yields the following
result.

\begin{theorem} \label{thoptnormal}
For any $\eps>0$, the adaptive wavelet methods from 
\cite{CDD2} 
applied to the normal equations \eqref{eq:normal} 
using above ${\bf APPLY}_{{\bf B}^{\ast} {\bf B}}$ and ${\bf
  RHS}_{{\bf B}^{\ast} {\bf f}}$ 
routines produce approximations ${\bf u}_{\eps}$ to ${\bf u}$ 
which satisfy $\|{\bf u}-{\bf u}_{\eps}\|_{\ell_2(\nabla^{X})} \leq \eps$. 
If for some $s>0$, ${\bf u} \in \cA^s_{\infty}(\ell_2(\nabla^{X}))$, 
then $\# {\rm supp}\,{\bf u}_{\eps} 
\lesssim \eps^{-1/s}\|{\bf u}\|_{\cA^s_{\infty}(\ell_2(\nabla^{X}))}^{1/s}$, 
with constant
only dependent on $s$ when it tends to $0$ or $\infty$, 
and on $\|{\bf B}\|$ and $\|{\bf B}^{-1}\|$ when they tend to infinity.

If $s < s^{\ast}$, then
the number of arithmetic operations and storage locations 
required by a call of either of these adaptive wavelet methods 
with tolerance $\eps>0$ is bounded by some multiple of
$$
1+\eps^{-1/s}(1+a_{{\bf B},s}(1+a_{{\bf B}^{\ast},s}))\|{\bf
  u}\|_{\cA^s_{\infty}(\ell_2(\nabla^{X}))}^{1/s} 
$$
where this multiple only depends on $s$ when it tends to $0$ or $\infty$, 
and on $\|{\bf B}\|$ and  $\|{\bf B}^{-1}\|$ when they tend to infinity.
\end{theorem}
\subsection{$s^{\ast}$-computability of ${\bf B}$ in \eqref{eq:GalMat}}
\label{sec:S1}
We apply the general concepts to the space-time variational
formulation \eqref{eq:ParaWeak} and the space-time tensor-product
wavelet bases 
$\Psi^X = \Theta^X\otimes \Sigma$
and 
$\Psi^Y = \Theta^Y \otimes \Sigma$ 
in Proposition~\ref{prop:RieszInX=Y}.  

Due to the discussion in Section~\ref{sec:XYchoice},
it suffices to show $s^{\ast}$-admissibility of both,
${\bf B}$ and ${\bf B}^{\ast}$,
for $s^{\ast} > s_{\mathrm{max}}$ with $s_{\mathrm{max}}$ as defined in \eqref{eq:sast}.
The bi-infinite matrix ${\bf B}$ defined in \eqref{eq:GalMat} comprises
of a sum of tensor products of bi-infinite matrices, each factor matrix
corresponding to either the 
Gram matrices 
$\langle \Theta_> , \Theta_> \rangle_{L^2(\IR_>)}$
or 
$ \langle \Sigma, \Sigma \rangle_{L^2(D)}$ 
or of the ``stiffness'' matrices 
$a([\Sigma]_V, [\Sigma]_V)$ 
with respect to the Riesz bases $\Theta_>$ and $\Sigma$ (cp.\ Section
\ref{sec:RieszInXY}). 

To apply the general theory of adaptive wavelet discretizations
of \cite{CDD1,CDD2,Stev09Surv}, the key step is the verification
of $s^*$-compressibility and of $s^*$-computability 
of the matrix ${\bf B}$ in  \eqref{eq:GalMat}.

We verify $s^{\ast}$-computability of
${\bf B}$ in \eqref{eq:GalMat} with the following
result \cite[Prop.~8.1]{ScSt09}.
\begin{proposition} \label{prop:ProdCptable}
Let for some $s^{\ast}>0$, ${\bf D}$, ${\bf E}$ be $s^{\ast}$-computable.
Then
\begin{enumerate}
\renewcommand{\theenumi}{\alph{enumi}}
\renewcommand{\labelenumi}{{\upshape(\alph{enumi})}}
\item \label{a}
${\bf D} \otimes {\bf E}$ is $s^{\ast}$-computable
with computability constant satisfying, for $0<\bar{s}<\tilde{s}<s^{\ast}$,
$c_{{\bf D} \otimes {\bf E},\bar{s}} 
\lesssim 
(c_{{\bf D},\tilde{s}} c_{{\bf E},\tilde{s}})^{\tilde{s}/\bar{s}}$
and
\item \label{b} for any $\eps \in(0,s^{\ast})$, ${\bf D} \otimes {\bf
    E}$ is $(s^{\ast}-\eps)$-computable, with computability constant
  $c_{{\bf D} \otimes {\bf E},\bar{s}}$ satisfying, for
  $0<\bar{s}<s^{\ast}-\eps<\tilde{s}<s^{\ast}$, $c_{{\bf D} \otimes
    {\bf E},\bar{s}} \lesssim \max(c_{{\bf D},\tilde{s}},1)
  \max(c_{{\bf E},\tilde{s}},1)$.
\end{enumerate}
The constants implicit by $\lesssim$ 
in the bounds on the computability constants in \eqref{a} and \eqref{b}
depend only on
$\tilde{s}$, $\tilde{s} \rightarrow \infty$ 
and on $\tilde{s}-\bar{s} \downarrow 0$.
\end{proposition}
We recall that we work under Assumption~\ref{asmp:s*compress},
so that the 
bi-infinite mass matrix 
$\matr{M} = \langle \Sigma, \Sigma \rangle_{L^2(D)}$
and the bi-infinite stiffness matrix
$\matr{A} = a([\Sigma]_V, [\Sigma]_V)$ 
are both
$s^*$ computable and compressible under our assumptions
\eqref{s1}--\eqref{s4}. 
%
%%%%%%%%%%%%%%%%%%%%%%%%%%%%%%%%%%%%%%%%%%%%%%%%%%%%%%%%%%%
\subsection{$s^*$-computability of the fractional time derivatives}
\label{sec:sCompThMat}
%%%%%%%%%%%%%%%%%%%%%%%%%%%%%%%%%%%%%%%%%%%%%%%%%%%%%%%%%%%
Proposition~\ref{prop:ProdCptable} and 
Assumption~\ref{asmp:s*compress} reduce the analysis of
$s^{\ast}$-compressibility of ${\bf B}$ in \eqref{eq:GalMat}
to the verification of the $s^{\ast}$-compressibility
of the temporal ``stiffness'' and ``mass'' matrices
\be\label{eq:ThetMat}
\matr{D} 
:= 
\big\langle {\rD^{\frac12}_+}[\Theta^X]_{\frac12}, 
{\rD^{\frac12}_-}[\Theta^Y]_{\frac12}\big\rangle_{L^2(\IR_>)}\;,
\qquad 
\matr{G} 
:= \big\langle\Theta^X , \Theta^Y \big\rangle_{L^2(\IR_>)}
\;,
\ee
i.e., 
on the compressibility of the
``stiffness'' matrix $\matr{D}$ 
and of the ``mass''-matrix $\matr{G}$ 
of the fractional time derivative in \eqref{eq:DefB}.

We discuss $s^*$-computability of $\matr{D}$ and $\matr{G}$ in the
sense of Definition~\ref{def:scomput}.  We assume at our disposal
Riesz bases $\Theta^X$ of $H^{\frac12}_{00,\{0\}}(\IR_>)$ and
$\Theta^Y$ of $H^{\frac12}(\IR_>)$ as in Section~\ref{sec:WavR+} and,
in particular, that properties \eqref{t1}--\eqref{t4} of that section
hold for elements of either of these bases.

The $s^*$-computability of  $\matr{G}$ follows as in 
\cite[Sect.~8.2]{ScSt09} from the properties 
\eqref{t1}--\eqref{t4} of $\Theta^X$ and $\Theta^Y$.
It remains to address $s^*$-computability 
of $\matr{D}$ in \eqref{eq:ThetMat}.

To this end, we observe that by a density argument,
Lemma~\ref{lem:intbyparts} and, in particular, 
the fractional integration by parts identity 
\eqref{eq:6b} remain valid for 
$w \in H^{\frac12}_{00,\{0\}}(\IR_>)$ and for $v \in H^{\frac12}(\IR_>)$.
Since  
$\Theta^X$ is a Riesz basis of $H^{\frac12}_{00,\{0\}}(\IR_>)$
and 
$\Theta^Y$ of $H^{\frac12}(\IR_>)$, we obtain
from \eqref{eq:6b} that 
$$
\matr{D} 
= 
\big\langle {\rD^{\frac12}_+}[\Theta^X]_{\frac12}, 
{\rD^{\frac12}_-}[\Theta^Y]_{\frac12}
\big\rangle_{L^2(\IR_>)}
=
\big\langle \rD [\Theta^X]_{1}, \Theta^Y \big\rangle_{L^2(\IR_>)}
\;.
$$
Now using properties \eqref{t1}--\eqref{t4} of the 
temporal wavelet bases $\Theta^X$ and $\Theta^Y$,
we establish $s^*$-computability of $\matr{D}$ 
as in \cite[Sect.~8.2]{ScSt09}.
\subsection{Optimality}
\label{sec:Opt}
The preceding considerations can be combined into 
\begin{theorem} \label{thopthighdim}
Consider the parabolic problem \eqref{eq:ParProb}, \eqref{eq:ParIC} 
in the weak form \eqref{eq:ParaWeak} 
with spatial bilinear form as in Section~\ref{sec:FuncSpc}.
Consider its representation ${\bf B} {\bf u}={\bf f}$ using 
temporal and spatial wavelet bases 
$\Theta$ and $\Sigma$ as above.

Then for any $\eps>0$, the adaptive wavelet 
methods from \cite{CDD2} applied to the normal equations 
\eqref{eq:normal}
produce an approximation ${\bf u}_{\eps}$ with 
$$
\|u-{\bf u}_{\eps}^\top [\Theta \otimes \Sigma]\|_{X}
\eqsim
\|{\bf u}-{\bf u}_{\eps}\|
\leq \eps\;.
$$

If for some $s>0$, 
${\bf u}\in \cA^s_{\infty}(\ell_2(\nabla^{X}))$, 
then 
${\rm supp}\,{\bf u}_{\eps} 
 \lesssim \eps^{-1/s} \|{\bf u}\|_{\cA^s_{\infty}(\ell_2(\nabla^{X}))}^{1/s}$, 
with the implied constant only dependent on $s$ when it tends to $0$
or $\infty$.  

If, for arbitrary $s^{\ast}>0$, it holds $s<s^{\ast}$, then the number
of operations and storage locations required by one call of the
space-time adaptive algorithm with tolerance $\varepsilon > 0$ is
bounded by some absolute multiple of
\begin{align*} 
\eps^{-1/s} n^2 \|{\bf u}\|_{\cA^s_{\infty}(\ell_2(\nabla^{X}))}^{1/s}+1
\;.
\end{align*}
Here, the implied constant depends only on the Riesz and the admissibility 
constants of the spatial wavelet bases $\Sigma$.
\end{theorem}

\section{Acknowledgement}
\label{sec:Ack}
This work was initiated during a visit of S.L. at
the research institute for mathematics (FIM) of ETH,
and continued during a visit of C.S. at Chalmers
and during workshops at the 
Mathematical Research Institute Oberwolfach.% from July 27--Aug 02, 2013.
\bibliographystyle{plain}
\bibliography{Para}
\end{document}